\documentclass[11pt]{article}
\usepackage{amssymb,amsmath,amsthm,amscd,amsfonts }
\usepackage{enumerate}
\usepackage{showkeys,srcltx}

\newtheorem{thm}{Theorem}[section]
\newtheorem{cor}[thm]{Corollary}
\newtheorem{lem}[thm]{Lemma}

\newtheorem{clm}{Claim}

\newtheorem{prop}[thm]{Proposition}

\theoremstyle{definition}
\newtheorem{defn}[thm]{Definition}

\numberwithin{equation}{section}

\newcommand{\mb}[1]{\mathbb{#1}}
\newcommand{\tb}[1]{\textbf{#1}}
\newcommand{\mf}[1]{\mathfrak{#1}}
\newcommand{\infgen}[1]{\langle\langle{#1}\rangle\rangle}
\newcommand{\hes}{\mathcal{H}}
\newcommand{\diam}{\operatorname{diam}}
\newcommand{\im}{\operatorname{im}}
\renewcommand{\phi}{\varphi}
\renewcommand{\int}{\operatorname{int}}

\begin{document}


\baselineskip=17pt






\date{}
\title{Fundamental groups of locally connected subsets of the plane}
\author{Greg Conner and Curt Kent}
\maketitle

\begin{abstract}
We show that every homomorphism from a one-dimensional Peano continuum to a planar Peano continuum is induced by a continuous map up to conjugation. We then prove that the topological structure of the space of points at which a planar Peano continuum is not semi-locally simply connected is determined solely by the algebraic structure of its fundamental group.  Furthermore, we demonstrate how to reconstruct the topological structure of the space of points at which a planar Peano continuum is not semi-locally simply connected using only the subgroup lattice of its fundamental group.
\end{abstract}



\section*{Introduction }

The Hawaiian earring is the one-point compactification of countably many disjoint open arcs. Its fundamental group, which we will denote by $\mathbb H$, is called the Hawaiian earring group.

The fundamental groups of planar and one-dimensional Peano continua have been studied extensively.  In \cite{hig}, G. Higman studies the inverse limit of finite rank free groups which he calls  the unrestricted free product of countably many copies of $\mb Z$. There he proves that this group is not free and that each of its free quotients has finite rank. He considers a subgroup $P$ of the unrestricted product which turns out to be a Hawaiian earring group but he does not prove it there.  In \cite{cf}, Curtis and Fort show that the Menger curve has fundamental group which embeds in the inverse limit of finite rank free groups.  As well in \cite{mm}, Morgan and Morrison study the Hawaiian earring group.  Cannon and Conner in \cite{cc1}and \cite{cc3} studied the fundamental groups of one-dimensional spaces. In \cite{at}, Akiyama, Dorfer, Thuswaldner, and Winkler studied the fundamental group of the Sierpinski-Gasket.

Katsuya Eda \cite{eda2} was the first to prove results concerning homomorphisms between the fundamental groups of one-dimensional Peano continua.  He proved that every homomorphism from $\mb H$ to $\mb H$ is conjugate to a homomorphism which is induced by a continuous map. Summers gave a combinatorial proof of the same result in \cite{cs}. In \cite{eda}, Eda was able to extend his original result to show that a homomorphism from $\mb H$ to the fundamental group of a one-dimensional metric space is conjugate to a homomorphism which is induced by a continuous map.

In Section \ref{sec1}, we prove that homomorphisms from $\mb H$ to the fundamental group of a planar Peano continuum are conjugate to homomorphisms induced by continuous maps.  This is then used to prove the following result.

\textbf{Theorem \ref{cont1dim}} \emph{ Let $X$ be a one-dimensional Peano continuum and $Y$ a planar Peano continuum.  Then for every homomorphism $\phi: \pi_1(X,x)\to \pi_1(Y,y_0)$ then there exists a path $\alpha: (I,0,1)\to (Y,y_0,y)$ and a continuous function $f:X\to Y$ such that $\hat \alpha\circ\phi = f_*$.}

Eda \cite{eda3} proved the analogue of Theorem \ref{cont1dim} when both $X$ and $Y$ are one-dimensional Peano continua.  In Section \ref{sec2}, we will show the necessary modifications of his proof to be able to extend it to the case when $Y$ is planar.  The second author has also shown that the theorem holds when $X$ is allowed to be a planar Peano continuum and $Y$ is either a planar or one-dimensional Peano continuum \cite{ck1}.

In \cite{eda}, Eda showed that the fundamental group of a one-dimensional continuum which is not semilocally simply connected at any point determines the homeomorphism type of the space.  Eda and Conner show that the fundamental group of a one-dimensional continuum which is not semilocally simply connected at any point can be used to reconstruct the original space with its topology \cite{ce}. We prove the following analogue for planar Peano continua.

\textbf{Theorem \ref{bad}} \emph{If $X$ is a planar Peano continuum, the
homeomorphism type of $B(X)$ is completely determined by $\pi_1(X,x_0)$,
where $B(X)$ is the subspace of points at which $X$ is not semilocally simply
connected.}

The proof that we present in Section \ref{sec3} is for a planar Peano continuum; however, with a few insignificant changes our proof works for both planar and one-dimensional Peano continua. Our proof shows that what is really necessary for a results of this type isn't the one-dimensionality of the spaces involved but the fact that essential loops are cannot be homotoped to much (see Lemma \ref{homotoped off} for an example of what this means in the planar case).  Thus types of results are possible for any class of spaces where essential loops have this property.

\section{Definitions}

A Peano continuum is a compact, connected, locally path connected, metric space.  For a metric space $(X,d)$, let $B_\epsilon^X(x_0) = \{x\in X \mid d(x, x_0)<\epsilon\}$ be the open $\epsilon$-ball about $x_0$ and $S_\epsilon^X(x_0) = \{x\in X \mid d(x, x_0)=\epsilon\}$ be the $\epsilon$-sphere about $x_0$.  When the underling space is understood, we will simply write $B_\epsilon(x_0)$ or $S_\epsilon(x_0)$.

The one-point compactification of a sequence of disjoint arcs can be realized in the plane as the union of circles centered at $(0,\frac1n)$ with radius $\frac1n$.  We will use $\textbf{E}$ to denote this subspace of the plane and $\mathbf{a}_n$ to denote the circle centered at $(0,\frac1n)$ with radius $\frac1n$.  Then $\pi_1(\textbf{E},(0,0)) = \mb H$. At times it will be convenient to consider certain subspaces of $\textbf{E}$ and certain subgroups of $\mb H$.  We will let $\textbf{E}_n = \bigcup\limits_{i\geq n} \mathbf{a}_i$, $\textbf{E}_e = \bigcup\limits_{i \text{ even}}\mathbf{a}_i$ and $\textbf{E}_o= \bigcup\limits_{i \text{ odd}}\mathbf{a}_i$.  Then let $\mb H_n$, $\mb H_e$, and $\mb H_o $ be their respective fundamental groups which are all isomorphic to $\mb H$.

Cannon and Conner have shown that $\mathbb H$ is generated in the sense of infinite products by a countable set of generators corresponding to the circles $\mathbf{a}_n$, where an infinite product is legal if each $\mathbf{a}_n$ is
transversed only finitely many times. For more details see \cite {cc1},\cite{cc3}. We will refer to this infinite generating set for the fundamental group of $\textbf{E}$ as $\{a_n\}$, i.e. $a_n$ represents the canonical path which transverses counterclockwise $\mathbf{a}_n$ one time.  We will frequently denote
the base point $(0,0)$ of $\textbf{E}$ by just $0$.

Let  $\{\mathbf b_n\}$ be a null sequence of loops based at $x_0$ in some metric space $X$.  Then any function $f: (\textbf{E},0)\to (X,x)$ which sends $\mathbf{a}_n$ continuously to $\mathbf b_n$ and sends $0$ to $x$ is continuous.  Often when defining functions from $\textbf{E}$ to $X$, we will only be interested in the property that $f_*(a_n)=[\mathbf b_n]$. Therefore to simplify the construction of such functions; when defining a function $f$ we will only state that $f$ maps $\mathbf{a}_n$ to $\mathbf b_n$ (for every $n$) meaning that $f$ maps $\mathbf{a}_n$ continuously to $ \mathbf b_n$ and sends $0$ to the base point of $ \mathbf b_n$ such that $f_*(a_n)=[\mathbf b_n]$.

We will use $[g]$ to represent the homotopy class relative to endpoints of the path $g$.  Then a continuous function $f$ induces an homomorphism on fundamental groups which we will denote by $f_*$. A loop \emph{$f$ is a reparametrization of $g$} if there exists continuous maps $\theta_1, \theta_2:[0,1]\to [0,1]$ such that $\theta_i$ is nondecreasing, $\theta_i(0)=0$ and $\theta_i(1)=1$ for $i=1,2$, and $f\circ\theta_1 = g\circ\theta_2$.  This is an equivalence relation and we will write $f=_pg$.

Suppose that $f_i:I\to X$ is a sequence of paths such that $f_i(1) = f_{i+1}(0)$.  Then $f_1*f_2$ will denote the standard concatenation of paths and $\prod_i f_i$ the infinite concatenation when it is defined.  We use $\overline f_1(t)$ to denote the path $\overline f_1(t) = f_1(1-t)$.  For a path $\alpha:(I,0,1)\to (X,x_0,x_1)$, let $\widehat \alpha: \pi_1(X, x_0) \to \pi_1(X, x_1)$ by the standard change of base point isomorphism, i.e. $\widehat \alpha([g]) = [\overline \alpha*g*\alpha]$.  This isomorphism has inverse $\widehat{\overline{\alpha}}$.

\begin{defn}
Let $X$ be a one-dimensional space.  We say that a path $f : I \to X$ is \emph{reducible} if there is an open arc $(x,y)\subset I$ such that $f|_{[x,y]}$ is a non-degenerate null-homotopic loop.  A loop $f$ is \emph{reduced} if it  is not reducible. A constant loop is, by definition reduced and every reparamerization  of a reduced path is reduced.  \end{defn}

\begin{thm}[Theorem 3.9 \cite{cc3}]\label{reducedloop-thm}
Let $f:I\to X$ be a continuous map into a one-dimensional space.  There exist a maximal (under set inclusion) set of open intervals $\bigl\{ (s_i, t_i)\bigr\}_{i\in J}$ with disjoint closures such that, for each $i\in J$, $f|_{[s_i,t_i]}$ is a null-homotopic subloop of $f$.

Define $f_r:I\to X$ by $f_r(t) = f(t)$ for $t\not\in \bigcup_{i\in J} (s_i, t_i)$ and $f_r(t) = f(s_i)$ for $t\in (s_i,t_i)$.  Then $f_r$ is the unique (up to reparametrization) reduced loop which is homotopic to $f$.
\end{thm}

We will use $f_r$ to denote the path where every null-homotopic subloop of $f$ is replaced by a constant path, in which case we will call $f_r$ a \emph{reduced representative} of $f$ or $[f]$.  When considering one-dimensional spaces, we will use $[\cdot]_r$ to denote a reduced representative of the path class of $[\cdot]$.

\begin{defn}Let $X$ be a one-dimensional space. Let $g:I\to X$ be a reduced representative for the path class $[g]$. Then we say that $a:I\to X$ is a \emph{head} for $g$ if there exists $b:I\to X $ such that $g =_p
a*b$. We will write $a \stackrel{h} \rightarrow g$. Additionally, we say that $b:I\to X$ is a \emph{tail} for $g$ if there exist $c:I\to X$ such that $g =_p c*\overline b$.  We will write $b\stackrel {t}\rightarrow g $.

We say that $t:I\to X$ is a \emph{head-tail} for a reduced path $g:I\to X$ if there exists $b:I\to X $ such that $g =_p
t*g*\overline t$, i.e. $t$ is a head and a tail for $g$.  We will write $t\stackrel{h-t} \longrightarrow g$.
\end{defn}

Since $g$ is a reduced path; the paths $a$, $b$, and $c$ are necessarily reduced paths.

\section{Homomorphisms from $\mb H$}\label{sec1}

We will use the following theorem for one-dimensional spaces.

\begin{thm}[Eda \cite{eda}]\label{hto1d}
Let $\phi:\mathbb H \to\pi_1(X,x_0)$ a homomorphism into the fundamental group of a one-dimensional Peano continuum $X$. Then there exists a continuous function $f:(\textbf{E},0) \to (X,x_1)$ and a path $T:(I,0,1)\to (X,x_0,x_1)$, with the property that $f_* =\widehat T \circ\phi$.  Additionally; if the image of $\phi$ is uncountable, then $T$ is unique up to homotopy rel endpoints.
\end{thm}

\subsection{Delineation}\label{delin}

To prove Theorem \ref{contplanar}, we will use an upper semicontinuous decomposition of the planar Peano continuum to get a continuous map into a one-dimensional Peano continuum which is injective on fundamental groups. This allows us to, in some degree, reduce the planar case to the one-dimensional case.  If $\pi_k$ is this decomposition map, we show how to lift the path $T$ such that $\widehat T\circ\pi_{k*}\circ\phi$ is induced be a continuous map. Then for $\alpha$ a lift of $T$, we show that $\widehat \alpha \circ \phi$ is induced by a continuous function.

\begin{defn}
Let $k$ be a line in the plane and $X$ a planar set. Let $\pi_k: X \to X/G$ be a decomposition map where the nontrivial decomposition elements of $G$ are the maximal line segments contained in $X$ which are parallel to $k$.  We will use $X_k$ to denote the decomposition space corresponding to $\pi_k$ and refer to $\pi_k$ as a \emph{delineation map}.
\end{defn}

The following is Theorem 1.4 in \cite{cc4}.

\begin{thm}[Cannon \& Conner]
If $X$ is a planar set and $k$ is any line in the plane; then $\pi_k$ is upper semicontinuous and $X_k$ is a one-dimensional Peano continuum.  In addition; if $X$ is a planar Peano continuum, then the induced homomorphism of fundamental groups is injective.

\end{thm}

Let $h: C\to D$ be a function.  Then a subset $C'$ of $C$ is \emph{$h$-saturated} if $C' = h^{-1}(D')$ for some $D'\subset D$.

\begin{lem}\label{continuous maps}
Let $X$ be a planar set and $\pi_k: X \to X_k$, $\pi_l: X \to X_l$ be delineation maps for non-parallel lines $k,l$ respectively.

If $f:Z\to X$ is a map from a first countable space $Z$ such that $\pi_k\circ f$ and $\pi_l\circ f$ are continuous, then $f$ is continuous.
\end{lem}

\begin{proof}
Let $k, l$ be non-parallel lines in the plane with slopes $m_k,m_l$ respectively.  Let $U$ be an open parallelogram in the plane with sides parallel to  $k$ and $l$.
Chose $c_{k,1}< c_{k,2}$ and $c_{l,1}< c_{l,2}$ such that $U = U_k\cap U_l$ where $U_k =\bigl\{ (x,y) \mid m_kx + c_{k,1} \leq y\leq m_kx + c_{k,2}\}$ and $U_l =\bigl\{ (x,y) \mid m_lx + c_{l,1} \leq y\leq m_lx + c_{l,2}\}$.

Since $U_k\cap X$ is $\pi_k$-saturated and open in $X$, $\pi_k(U_k\cap X)$ is open in $X_k$.  Similarly, $\pi_l(U_l\cap X)$ is open in $X_l$.  Then

\begin{align*}
 f^{-1} (U \cap X) &= f^{-1}(U_k\cap X)\cap f^{-1}(U_l\cap X)\\
&= \bigl(\pi_k\circ f\bigr)^{-1}\bigl( \pi_k(U\cap X)\bigr)\cap\bigl(\pi_l\circ f\bigr)^{-1}\bigl( \pi_l(U\cap X)\bigr)
\end{align*}

which is open since $\pi_k\circ f$ and $\pi_l\circ f$ are continuous.

This completes the proof since such sets from a basis for the topology of $X$.
\end{proof}

\begin{cor}\label{continuous maps-corrollary}
Let $X$ be a planar set and $\pi_k: X \to X_k$, $\pi_l: X \to X_l$ be delineation maps for non-parallel lines $k,l$ respectively.

A  sequence $x_n$ of points in $X$ converges to $x_0$ if and only if  $\pi_k(x_n)$ and $\pi_l(x_n)$ converge to $\pi_k(x_0)$ and $\pi_l(x_0)$, respectively.

\end{cor}

\begin{proof}
The one direction is trivial since delineation maps are continuous.  Suppose that $\pi_k(x_n)$ and $\pi_l(x_n)$ converge to $\pi_k(x_0)$ and $\pi_l(x_0)$, respectively.  Let $Z= \{0\}\cup\{ \frac 1n\}$ and $f: Z\to X$ such that $f(\frac 1i) = x_i$ and $f(0) = x_0$.  Then $\pi_k\circ f$ and $\pi_l\circ f$ are continuous.  Hence by $f$ is continuous and $x_n $ converges to $x_0$.

\end{proof}

\begin{lem} \label{kreduced}
Let $X$ be a planar Peano continuum.  If $g:I\to X$ is a path and $\pi_k\circ g$ has reduced representative $f$ then there exists $\tilde g:I\to X$ such that $\pi_k\circ \tilde g =_p f$ and $g$ is homotopic rel-endpoints to $\tilde g$.
\end{lem}

\begin{proof}  If $\pi_k\circ g$ is reduced, we are done.  Otherwise there exists an interval $[c,d]$ such that $\pi_k\circ g\bigl|_{[c,d]}$ is a non-constant null-homotopic subloop.  Then $\pi_k\circ g(c) = \pi_k\circ g(d)$ which implies that the line segment $\overline{g(c)g(d)}$ is contained in $X$. The loop $g\bigl|_{[c,d]}*\overline{g(d)g(c)}$ maps to $\pi_k\circ
g\bigl|_{[c,d]}$ and hence must be null-homotopic since $\pi_{k*}$ is injective. Therefore $g$ is homotopic to $ g'$ where the subpath $g\bigl|_{[c,d]}$ is replaced by the line segment $\overline{g(c)g(d)}$.

Suppose that $\bigl\{J_i = (c_i,d_i)\bigr\}$ is a maximal set of disjoint open subintervals of $I$ such that $\pi_k\circ g\bigl|_{\overline{J_i}}$ is a non-constant null-homotopic loop.  Let $l_i$ be a parametrization of the line segment from $g(c_i)$ to $g(d_i)$.  Let  $\tilde g$ be the path $g$ where each subpath $g|_{[c_i,d_i]}$ is replace by $l_i$. Then $\tilde g$ defines a continuous path such that $g(t) =\tilde g(t)$ for $t\not\in \cup_i J_i$.

We need to show that $g$ is homotopic to $\tilde g$.  Since $g$ is uniformly continuous, $\diam{\bigr(g \bigl|_{[c_i,d_i]}\bigl )}$ must converge to zero.

By Lemma \ref{cutoff}, there exists homotopies $H_i:I\times[c_i,d_i] \to X$ with the property that $H_i\bigr|_{\{0\}\times[c_i,d_i]} = g\bigr|_{[c_i,d_i]}$, $H_i\bigr|_{\{1\}\times[c_i,d_i]} = l_i$, and $\diam{H_i(I\times[c_i,d_i])} \to 0$.

Lemma \ref{aalem1} implies that $g$ is homotopic to $\tilde g$.  By our choice of $\{J_i\}$, $\pi_k\circ\tilde g$ is reduced which completes the proof.
\end{proof}

\begin{defn}If $g$ maps to a reduced path under $\pi_k$, we will say that $g$ is reduced with respect to $k$ or $g$ is $k$-reduced.  For any path $g$; if $\tilde g$ is obtained from $g$ as in Lemma \ref{kreduced} then we will say $\tilde g$ is obtained by reducing $g$ with respect to $k$.\end{defn}

\begin{lem}\label{freely_k_reduced-lem}
If $g:S^1\to X$ is a loop and $\pi_k\circ g$ has cyclically reduced representative $f$ then there exists $\tilde g:S^1\to X$ such that $\pi_k\circ \tilde g =_p f$ and $g$ is freely homotopic to $\tilde g$.
\end{lem}

\begin{proof}
The proof is the same as in Lemma \ref{kreduced}, except we choose $\bigl\{J_i = (c_i,d_i)\bigr\}$ to be a maximal set of disjoint open subsets of $S^1$ which are homeomorphic to open intervals such that $\pi_k\circ g\bigl|_{\overline{J_i}}$ is a non-constant null-homotopic loop.
\end{proof}

\subsection{Weight Function}

Before we proceed, we need to introduce a weight function.  This weight function is a discrete version of the oscillation function defined by Cannon, Conner, and Zastrow in \cite{ccz}.  The discrete version was also used
in \cite{ck}, since it is preserved under nerve approximations.

\subsubsection{Weight function}
Let $X$ be a normal topological space. For a path $f: I\to X$ and $U, V$ disjoint open subsets of $X$,
let $r_f :f^{-1} (U\cup V) \to \{-1,1\}$ by $r_f (s) = 1$ if $f(s)\in U$ and $r_f (s) = -1$ if $f(s)\in V$. Let $\overline w^V_U (f)= \sup\big(- \sum\limits_i r_f(s_i)\cdot r_f(s_{i+1})\big)$ taken over all increasing countable subsets of $f^{-1} (U\cup V)$.  For any collection consisting of 0 or 1 point, we will consider the sum to be 0.

If the image of two consecutive points in our countable subset of $f^{-1} (U\cup V)$ are contained in the same open set, then the sum would increase by deleting one.  Thus the supremum is obtained by choosing an increasing sequence of points from $f^{-1} (U\cup V)$ whose image alternates between $U$ and $V$. Therefore $\overline w$ counts the number of times that the image of $f$ alternates between $U$ and $V$. If $f$ is continuous and $U$, $V$ have disjoint closures, then its image is compact and can only alternate between sets with disjoint closures finitely many times.  So the supremum is actually realized for some finite set of points. If $U'\subset U$ and $V'\subset V$, then $\overline w^{V'}_{U'} (f)\leq \overline w^V_U (f)$.

\begin{defn}
{The \emph{weight} of $f$ with respect to subsets $A$ and $B$ of
$X$ with disjoint closures is $w^A_B (f) = \inf \overline w^V_U
(f)$ taken over all possible separations $U$ and $V$ of
$\overline A$ and $\overline B$. If $[f]$ is a homotopy
equivalence class of functions, then $w^A_B([f]) =
\inf\limits_{f\sim f'}\{w^A_B (f')\}$.}
\end{defn}

Suppose that $\theta:I \to I$ is nondecreasing function such that $\theta(0)=0$ and $\theta(1)=1$.  Then  $\overline w^U_V(f) =\overline w^U_V(f\circ\theta)$.  Thus $\overline w$ is preserved under reparameterization.

If $f:I\to X$ is a map into a one-dimensional space and $f\bigl|_{\overline J_i}$ is a null-homotopic subloop.  Then $\overline w^U_V(f) \geq\overline w^U_V(f')$ where $f'$ is the path obtained by replacing the subpath  $f\bigl|_{\overline J_i}$ of $f$ by a constant subpath.  Thus we obtain $\overline w^U_V([f]) =\overline w^U_V(f_r)$.

The set $\{\overline w^U_V (f) \ | \ U, V \text{are a separation of } \overline A, \overline B\}$ is a subset of the natural numbers and hence has a minimum. Thus we may chose an open separation $U$, $V$ such that $w^A_B (f) = \overline w^V_U (f)$. For continuous $f$, $f^{-1}(U \cup V)$ can be partitioned into a finite collection of
disjoint open sets, $I_i$, in $I$ with a natural ordering ($I_i \leq I_j$ if $x\leq y $ for all $x\in I_i$ and $y\in I_J$) such that $f(I_i)\subset U$ or $f(I_i)\subset V$. If for some $i$, $f(I_i)\subset U\backslash \overline A$ (or $f(I_i)\subset V\backslash \overline B$), then there would exist an open set contained in $U$ (or $V$) containing $ \overline A$ (or $\overline B$)  which did not intersect $f(I_i)$ and thus alternate fewer times. Therefore, $f(I_i)$
must intersect $\overline A$ (or $\overline B$). So points which realize the weight can be chosen in the closures of $A$ and $B$. Thus there exists a finite increasing set of points $\{s_i\}$ which can be chosen to have image in the closures of $A$ and $B$ such that $w^A_B (f) = \overline w^V_U (f) = \sum\limits_i -r_f(s_i)\cdot r_f(s_{i+1})$.  We will sometimes write $w^A_B (f) = \sum\limits_i -r_f(s_i)\cdot r_f(s_{i+1})$.  This implicitly implies a choice of $U$ and $V$ to define $r_f$. However; if $\{s_i\}$ are chosen to have image in the closure of $A$ and $B$, $r_f(b_i)$ is the same for every choice of $U$ and $V$.  Thus we will ignore this choice at times.


\begin{lem}\label{preserve0}
Let $X$ be a planar Peano continuum and $k$ a line in the plane.  Suppose that $A$ and $B$ are disjoint closed $\pi_k$-saturated sets and $A_k$,$B_k$ their respective images under $\pi_k$. Then $w^A_B(g) = w^{A_k}_{B_k} (\pi_k\circ g)$.
\end{lem}

This follows directly from the fact that the weight can be realized by a finite set of points.

\begin{lem} \label{lem6}
If $g$ is an essential loop in a one-dimensional space $X$, there exist sets $O',O''$ with disjoint closures such that, for all $r$, $w^{O'}_{O''}([g]^r) \geq r$.  In addition; if $X = \pi_k(Y)$ for some planar Peano continuum $Y$ and some line $k$ in the plane, then $O'$ and $0''$ can be chosen to be the images of disjoint closed half planes of $Y$ with boundaries parallel to $k$.
\end{lem}

\begin{proof}  We may assume $g$ is a reduced path since the weight of the reduced path is less than or equal to
the weight of all paths in its homotopy class.

Since the set of head-tails for $g$ has a natural ordering which is bounded there exists a maximal head-tail, $t$ for $g$ where $g = t*f*\overline t$ such that $f*f$ is reduced. Hence $w^A_B\bigl([g]^r\bigr)\geq w^A_B\bigl([f]^r\bigr) = r\Bigl(w^A_B\bigl([f]\bigr)\Bigr)$ for all $A$ and $B$. Since $g$ is essential, $f$ is essential which implies $f$ is a non-degenerate path. Then for any two distinct points $O'$ and $O''$ in the image of $f$, we have $w^{O'}_{O''}\bigl([f]\bigr)\neq 0$.  Thus $w^{O'}_{O''}\bigl([g]^r\bigr)\geq r\Bigl(w^{O'}_{O''}\bigl([f]\bigr)\Bigr) \geq r$, for any $r$.

Suppose that $X = \pi_k(Y)$ for some planar Peano continuum $Y$ and some line $k$ in the plane.  Then $\pi_k^{-1}(O'), \pi_k^{-1}(O'')$ are disjoint line segments in the plane.  Hence there exits disjoint closed half planes $A, B$ such that $\pi_k^{-1}(O')\subset A$ and $ \pi_k^{-1}(O'')\subset B$.  Thus $w^{\pi_k(A)}_{\pi_k(B)}\bigl([g]^r\bigr)\geq r\Bigl(w^{\pi_k(A)}_{\pi_k(B)}\bigl([f]\bigr)\Bigr)\geq r$, for any $r$.
\end{proof}

For $f$ and $g$ as in the proof of Lemma \ref{lem6}, we will say that $f$ is the \emph{core} of $g$.

\begin{lem}{ \label{preserve}
Let $X$ be a planar Peano continuum and $k$ a line in the plane.  The delineation map, $\pi_{k}$, preserve weights of homotopy classes on disjoint closed saturated sets, i.e. $w^A_B([g]) = w^{A_k}_{B_k} ([\pi_k\circ g])$.}
\end{lem}

\begin{proof}
Let $\tilde g$ be the path homotopic to $g$ such that
$\pi_k\circ\tilde g$ is reduced.  Then $w^{A_k}_{B_k}([\pi_k\circ
g])  = w^{A_k}_{B_k}(\pi_k\circ \tilde g)= w^A_B(\tilde g) \geq
w^A_B([g])$. For $g'$ homotopic to $g$, $w^A_B(g')=
w^{A_k}_{B_k}(\pi_k\circ g') \geq w^{A_k}_{B_k}(\pi_k\circ
 \tilde g)= w^{A_k}_{B_k}([\pi_k\circ g])$. Then $w^A_B([g])\geq
 w^{A_k}_{B_k}([\pi_k\circ g])$.

Thus $w^A_B([g]) = w^{A_k}_{B_k}([\pi_k\circ g])$.

\end{proof}

Lemma \ref{preserve} implies that a necessary condition for $g$ to be $k$-reduced is that it have minimal weight in its path class with respect to all disjoint half-planes $A$ and $B$ with boundaries
parallel to $k$.

In fact this condition is also sufficient. Suppose $g$ has minimal weight in its path class with respect to all subsets $A$ and $B$ with boundaries parallel to $k$. If $\pi_k\circ g$ is not reduced, then there exists $g(c)$ and $g(d)$ such that $\overline{g(c)g(d)}$ is in contained in $X$ and $\pi_k\circ g\bigr|_{[c,d]}$ is null-homotopic but not constant.
Then $g\bigr|_{[c,d]}$ must not be contained in the line segment $\overline{g(c)g(d)}$. However, then $g$ is homotopic to $\tilde g$ where $g\bigr|_{[c,d]}$ is replace by $\overline{g(c)g(d)}$ and the weight of $\tilde g$ is strictly less than the weight of $g$ for some disjoint half-planes with boundaries parallel to $k$.

This characterization of begin reduced implies that if $\tilde g$ is obtained by reducing $g$, a $k$-reduced path, with respect to $l$; then $\tilde g$ is $(k,l)$-reduced (reduced with respect to both $k$ and $l$).

\begin{lem}
Suppose that $f:[0,3]\to Y$ is a path into a one-dimensional space $Y$ for $i\in 0,1,2$.  Then for any $A$ and $B$ with disjoint closures, $w^A_B \bigl([f]\bigr)\leq w^A_B \bigl(f|_{[0,1]}\bigr) +w^A_B \bigl(f|_{[1,2]}\bigr) + w^A_B \bigl(f|_{[2,3]}\bigr)+2$.
\end{lem}

\begin{proof}
Fix a set of maximal disjoint open intervals $\bigr\{(a_i,b_i)\bigl\}_{i\in J}$ such that $f|_{[a_i, b_i]}$ is a non-degenerate null homotopic loop.  Then define $f_r$ as in Theorem \ref{reducedloop-thm}.  Choose $\{s_i\}$ such that $w^A_B (f_r) = \sum\limits_{j=1}^{n} -r_{f_r}(s_j)\cdot r_{f_r}(s_{j+1})$.  Since $f_r$ is constant an the intervals $(a_i, b_i)$, we may assume that $s_j \not\in \bigcup\limits_{i\in J} (a_i, b_i)$ for all $j$.  Then $ \sum\limits_{j=1}^{n} -r_{f_r}(s_j)\cdot r_{f_r}(s_{j+1}) = \sum\limits_{j=1}^{n} -r_{f}(s_j)\cdot r_{f}(s_{j+1})$ and

\begin{align*}
\sum\limits_{j=1}^{n} -r_{f}(s_j)\cdot r_{f}(s_{j+1}) & \leq  \sum\limits_{s_{j+1}\leq 1} - r_{f}(s_j)\cdot r_{f}(s_{j+1}) +  \sum\limits_{s_j\geq 1\atop \newline s_{j+1}\leq 2} -r_{f}(s_j)\cdot r_{f}(s_{j+1}) \\
 &\hspace{2cm} + \sum\limits_{s_j\geq 2} -r_{f}(s_j)\cdot r_{f}(s_{j+1}) + 2 \\
&\leq  w^A_B \bigl(f|_{[0,1]}\bigr) +w^A_B \bigl(f|_{[1,2]}\bigr) + w^A_B \bigl(f|_{[2,3]}\bigr)+2
\end{align*}

\end{proof}

\begin{lem}\label{different delineations}
Suppose that $\alpha, \beta :I \to X$ are $(k,l)$-reduced paths such that $\pi_k \circ \alpha =_p \pi_k \circ \beta$.  Then $w^A_B(\alpha)\leq w^A_B(\beta)+4$ for any  $A$ and $B$ which are disjoint closed half-planes with boundary parallel to $l$.
\end{lem}

\begin{proof}
Since $\pi_k \circ \alpha =_p \pi_k \circ \beta$, the line segments $\overline{\alpha(0)\beta(0)}$ and $\overline{\beta(1)\alpha(1)}$ are contained in $X$.

\begin{clm}
$\overline{\alpha(0)\beta(0)}$ and $\overline{\beta(1)\alpha(1)}$ are $(k,l)$-reduced paths and $w^A_B(\overline{\alpha(0)\beta(0)}),w^A_B(\overline{\beta(1)\alpha(1)}) \leq 1$ for any  $A$ and$ B$ which are disjoint closed half-planes with boundary parallel to $l$.
\end{clm}

\begin{proof}[Proof of claim.]
Notice that $\overline{\alpha(0)\beta(0)}$ and $\overline{\beta(1)\alpha(1)}$ are line segments in $X$ parallel to $k$; hence, their images in $X_k$ are degenerate and reduced.  Any two non-parallel line segments intersect at most at a single point.  Thus the line segments are also $l$-reduced and the weight condition follows.
\end{proof}

However, $\overline{\alpha(0)\beta(0)}*\beta * \overline{\beta(1)\alpha(1)}$ is not necessarily $l$-reduced.

Fix $A$ and $ B$ which are disjoint closed half-planes with boundary parallel to $l$.  Define $f:[0,3]\to X$ by $f|_{[0,1]} =_p \overline{\alpha(0)\beta(0)}$, $f|_{[1,2]} =_p \beta$, and $f|_{[2,3]} =_p \overline{\beta(1)\alpha(1)}$.  The previous lemma together with the claim gives us the desired inequality.
\end{proof}

\subsection{Induced by a continuous map}

For $\phi:\mathbb H \to \pi_1(X,x_0)$ a fixed homomorphism into the
fundamental group of a planar Peano continuum and each line $k$ in the plane, we will use $T_k$ to denote the path such that $\widehat {T}_k \circ (\pi_{k*}\circ\phi)$ is induced by a continuous map.

The key to being able to reduced the planar case to the
one-dimensional case is the following proposition.

\begin{prop}\label{pheadtail} {For $k$ a line in the plane, there
exists $\alpha_k$ a path in $X$ such that $\pi_k(\alpha_k) = T_k$.}
\end{prop}

To prove this proposition we will construct a single word
$a\in\mathbb H$ such that $T_k\stackrel{t}\rightarrow
(\pi_k\circ\phi(a))_r$.

\begin{lem}\label{toobig1}
Let $X$ be a one-dimensional space. Suppose $\alpha,\beta,\gamma:[0,1]\to X$ are reduced paths such that $\alpha(1)= \beta(0)$, $\beta(1) = \gamma(0)$, and $\beta*\gamma$ is a reduced path.  If there exists $A$, $B$ such that $w^{A}_{B}(\alpha) < w^{A}_{B} (\beta)$ then $(\alpha*\beta)_r*\gamma$ is a reduced path.
\end{lem}

\begin{proof}
The only way $(\alpha*\beta)_r*\gamma$ might not be a reduced path is if $\alpha = \gamma' * \overline \beta $ for some non-constant path $\gamma'$.  This would imply $w^{A}_{B}(\beta) \leq w^{A}_{B} (\alpha)$ for all $A$, $B$.
\end{proof}

\begin{lem}\label{toobig2}
Let $X$ be a one-dimensional space. Suppose $\alpha,\beta,\gamma:[0,1]\to X$ are reduced paths such that $\alpha(1)= \beta(0)$, $\beta(1) = \gamma(0)$, and $\alpha*\beta$ is a reduced path. If  $\diam(\gamma) < \diam(\beta)$ then $\alpha*(\beta*\gamma)_r$ is a reduced path.
\end{lem}

The proof follows is similar to the proof of Lemma \ref{toobig1}.

In Lemma \ref{toobig1}, the condition that there exists $A$, $B$ such that $w^{A}_{B}(\alpha) < w^{A}_{B} (\beta)$ can be replaced by $\diam(\alpha)< \diam(\beta)$.  The analogous change for Lemma \ref{toobig2} is also valid.

\begin{lem}\label{key}
Let $Y$ be a one-dimensional Peano continuum.  Let $f:\textbf{E}\to Y $ be
a continuous function enjoying the property that $\im(f_*)$ is uncountable
and let $T$ be any path from $f(0)$ to $y\in Y$. Then there exists $a\in\mathbb H$ such that no non-constant head of $T$ is a tail for $f_*(a)_r$, i.e. $f_*(a)_r * T$ is reduced.
\end{lem}

This can be derived as a result of Lemma 6.1 in \cite {eda}.  However; the proof of Proposition \ref{pmaximalheadtail} depends heavily on the construction of $a$, so we will present a proof here for completeness.

\begin{proof}
Let $b_n$ and $t_n$ be subpaths of $f_*(a_n)_r$ such that $f_*(a_n)_r =_p t_n* b_n*\overline{t}_n$ where $b_n$ is the core of $f_*(a_n)_r $. We may assume that $b_n$ is an essential loop, by passing to a subsequence if necessary.

Since $f$ is continuous, the paths $t_n$ and $f_*(a_n)_r$ must both form null sequences.  Hence, we may choose an increasing sequence $i_n$ such that  $\diam{(t_{i_n})}$ is decreasing and $\diam{(f_*(a_n)_r)} < \dfrac{\diam{(t_{i_{n-1}})}}2$.

There exists an $r_n$ such that for some $A_n$ and $ B_n$, $w^{A_n}_{B_n}(T) < w^{A_n}_{B_n}(b_{i_n}*\cdots* b_{i_n})$, the concatenation of $r_n$ copies of $b_{i_n}$.

We will inductively define integers $s_n$ and homotopy classes $c_n$ in $\mathbb H$.  Let $s_1 = 0$,  $c_1$ be the identity in $\mathbb H$, and suppose $s_i$ and $c_i$ are defined for $1\leq i<n$.

Again, there exists an $s_n$ such that for some $A'_n$ and $ B'_n$, $w^{A'_n}_{B'_n}(f_*(\prod\limits_{i=1}^{n-1}(c_i))_r*t_{i_n}) < w^{A'_n}_{B'_n}(b_{i_n}*\cdots* b_{i_n})$, the concatenation of $s_n$ copies of $b_{i_n}$.  Let $c_n = a_{i_n}^{r_n+s_n}$.  Then $f_*(c_n)_r = t_n* b_n^{r_n+s_n}*\overline{t}_n$.  To simplify notation, let $t_{i_0}, b_{i_0}$ be constant paths and $r_{i_0}, s_{i_0}=0$.

Let $a = \prod\limits_{n=1}^\infty c_n$ and note that this is a well-define homotopy class in $\mathbb H$.  Since $f$ is a continuous maps; $ \prod\limits_{n=1}^\infty t_n* b_n^{r_n+s_n}*\overline{t}_n$ is a well defined path and  $f_*(a) = \Bigl[\prod\limits_{n=1}^\infty t_n* b_n^{r_n+s_n}*\overline{t}_n\Bigr]$.  Our objective with the next three claims is to understand the backtracking that appears in the path $ \prod\limits_{n=1}^\infty t_n* b_n^{r_n+s_n}*\overline{t}_n$.

\begin{clm}\label{claim1}
For $m\geq 2$, \begin{equation}\label{1}  t_{i_1}*b_{i_1}^{s_1} *\Bigl(\prod\limits_{j=2}^{m} b_{i_{j-1}}^{r_{j-1}} * (\overline t_{i_{j-1}} * t_{i_j} * b_{i_j}^{s_{i_j}})_r \Bigr) \end{equation} is a reduced path.
\end{clm}

Since $t_{i_0},b_{i_0}$ are constant paths and $r_{i_0}=0$, Equation (\ref{1}) is the same (up to reparameterization) as
 $$ \prod\limits_{j=1}^{m} b_{i_{j-1}}^{r_{j-1}} * (\overline t_{i_{j-1}} * t_{i_j} * b_{i_j}^{s_{i_j}})_r.$$

\begin{proof}[Proof of Claim \ref{claim1}]
We will proceed by induction on $m$.  The case $m=1$ is trivial. By inducting on $m$, suppose that  that $\prod\limits_{j=1}^{m} b_{i_{j-1}}^{r_{j-1}} * (\overline t_{i_{j-1}} * t_{i_j} * b_{i_j}^{s_{i_j}})_r $ is reduced.

Then $$\Bigl(\prod\limits_{j=1}^m b_{i_{j-1}}^{r_{j-1}} * (\overline t_{i_{j-1}} * t_{i_j} * b_{i_j}^{s_{i_j}})_r \Bigr) * b_{i_m}^{r_{i_m}} * \overline t_{i_m} $$ is a reduced path by Lemma \ref{toobig1} with $\alpha =\Bigl(\Bigl(\prod\limits_{j=1}^{m-1} b_{i_{j-1}}^{r_{j-1}} * (\overline t_{i_{j-1}} * t_{i_j} * b_{i_j}^{s_{i_j}})_r \Bigr)* b_{i_{m-1}}^{r_{m-1}}*\overline t_{i_{m-1}}*t_{i_{m}}\Bigr)_r $, $\beta = b_{i_{m}}^{s_{i_{m}}}$, and $\gamma= b_{i_m}^{r_{i_m}}*\overline t_{i_m}$.  Here we are also using the inductive hypothesis.

We can complete the proof by applying Lemma \ref{toobig2}, one more time, with \\ $\alpha = \Bigl(\prod\limits_{j=1}^m b_{i_{j-1}}^{r_{j-1}} * (\overline t_{i_{j-1}} * t_{i_j} * b_{i_j}^{s_{i_j}})_r \Bigr) * b_{i_m}^{r_{i_m}}$, $\beta = \overline t_{i_m}$, and $\gamma= t_{i_2}*b_{i_2}^{s_{i_2}}$.

\end{proof}

\begin{clm}\label{claim3}For $m\geq 2$,
\begin{equation}\label{4}b_{i_m}^{r_{i_m}} * \overline t_{i_m} * f_*\bigl(\prod\limits_{i={m+1}}^\infty c_i\bigr)_r\end{equation} is a reduced path.
\end{clm}

This follow by applying Lemma \ref{toobig2} with $\alpha = b_{i_m}^{r_{i_m}}$, $\beta = t_{i_m}$ and $\gamma = f_*\bigl(\prod\limits_{i={m+1}}^\infty c_i\bigr)_r$.

\begin{clm}\label{claim2} For all $m\geq 2$,
\begin{equation}\label{2}f_*(a)_r =_p t_{i_1}*b_{i_1}^{s_1} *\Bigl(\prod\limits_{j=2}^m b_{i_{j-1}}^{r_{j-1}} * (\overline t_{i_{j-1}} * t_{i_j} * b_{i_j}^{s_{i_j}})_r  \Bigr)* b_{i_m}^{r_{i_m}} * \overline t_{i_m} * f_*\bigl(\prod\limits_{i={m+1}}^\infty c_i\bigr)_r  \end{equation} where the equality implies that the right hand side is reduced.

\end{clm}

\begin{proof}[Proof of Claim \ref{claim2}.]
Since $ b_{i_m}$ is cyclically reduced and Equation (\ref{4}) is reduced,
$$b_{i_{m}}^{s_{i_{m}}} * \Bigl (b_{i_{m}}^{r_{i_{m}}} * \overline t_{i_{m}} * f_*\bigl(\prod\limits_{i={m+1}}^\infty c_i\bigr)\Bigr)_r $$ is a reduced path.

This give us that \begin{equation}\label{3} \Bigl(t_{i_1}*b_{i_1}^{s_1} *\Bigl[\prod\limits_{j=2}^{m} \bigl(b_{i_{j-1}}^{r_{j-1}} * (\overline t_{i_{j-1}} * t_{i_j} * b_{i_j}^{s_{i_j}})_r \Bigr]\Bigr)_r * \Bigl(b_{i_m}^{r_{i_m}} * \overline t_{i_m} * f_*\bigl(\prod\limits_{i={m+1}}^\infty c_i\bigr)\Bigr)_r \end{equation} is a reduced path by applying Lemma \ref{toobig1} with \\ $\alpha = \Bigl(t_{i_1}*b_{i_1}^{s_1} *\Bigl(\prod\limits_{j=2}^{m-1} \bigl(b_{i_{j-1}}^{r_{j-1}} * (\overline t_{i_{j-1}} * t_{i_j} * b_{i_j}^{s_{i_j}})_r \Bigr)*b_{i_{m-1}}^{r_{m-1}}*\overline t_{i_{m-1}} * t_{i_m}\Bigr)_r$, $\beta = b_{i_{m}}^{s_{i_{m}}} $, and \\ $\gamma=  \Bigl(b_{i_{m}}^{r_{i_{m}}} * \overline t_{i_{m}} * f_*\bigl(\prod\limits_{i={m+1}}^\infty c_i\bigr)\Bigr)_r  $.

Since Equations (\ref{1}), (\ref{4}), and (\ref{3}) are reduced; Equation (\ref{2}) is reduced.

\end{proof}

We can now complete the proof of Lemma \ref{key}.  Suppose that $T|_{[0,s]}$ is a non-constant tail of $f_*(a)_r$.  Then we can choose ${m_0}$ sufficiently large such that $\diam\Bigl(b_{i_{m_0-1}}^{r_{i_{m_0-1}}} * \overline t_{i_{m_0-1}} * f_*\bigl(\prod\limits_{i={{m_0}}}^\infty c_i\bigr)\Bigr) < \diam(T|_{[0,s]})$.  However, then Claim \ref{claim2} for $m= {m_0}$ would imply that $b_{m_0}^{r_{m_0}}$ is a subpath of $T$ which contradicts our choice of $r_{m_0}$.

\end{proof}

\begin{cor}\label{corkey}Let $X$ be a planar Peano continuum and $\phi: \mathbb H \to \pi_1(X, x_0)$ be a homomorphism with uncountable image. If $k$ is a line in the plane and $\widehat T_k\circ \pi_{k*}\circ\phi = f_*$ for some continuous map $f:\textbf{E} \to X_k$, then  $T_k$ is a tail for $\bigl(\pi_{k*}\circ \phi(a)\bigr)_r$ where $a$ is chosen as in Lemma \ref{key}.
\end{cor}

\begin{proof}
Note that $(\pi_{k*}\circ \phi(a))_r = (T_k*f_*(a)_r*\overline T_k)_r$.
Then Lemma \ref{toobig1} and Proposition \ref{key} imply that $(T_k*f_*(a))_r* \overline T_k $ is a reduced path.

\end{proof}

Then Proposition \ref{pheadtail} follows from Corollary \ref{corkey}
and Lemma \ref{kreduced}.

\begin{prop}\label{pmaximalheadtail}Let $X$ be a planar Peano continuum, $\phi: \mathbb H \to \pi_1(X, x_0)$ be a homomorphism with uncountable image, and $k_1, k_2$ lines in the plane. Then there exists a path $\alpha$ in $X$ such that $\pi_{k}(\alpha) = T_{k}$ where $T_{k}$ is the path such that $\widehat T_{k}\circ \pi_{k*}\circ\phi$ is induced by a continuous map and $k\in \{k_1,k_2\}$.
\end{prop}

\begin{proof}
If $k_1, k_2$ are parallel this is an immediate consequence of Corollary \ref{corkey}.  Hence we will assume that $k_1, k_2$ are non-parallel.

Let $f_j$ by the continuous map such that  $f_{j*} = \widehat T_{k_j} \circ (\pi_{{k_j}*}\circ\phi)$.

We will now show how to modify the proof of Proposition \ref{key} to find a single $a$ such that $T_{k_j}$ is tail for $\bigl(\pi_{{k_j}*}\circ \phi(a)\bigr)_r$ for either $j=1,2$.

Let $b_{n,j}$ and $t_{n,j}$ be subpaths of $f_{j*}(a_n)_r$ such that $f_{j*}(a_n)_r =_p t_{n,j}* b_{n,j}*\overline{t}_{n,j}$  where $b_{n,j}$ is the core of $f_{j*}(a_n)_r $.

We may assume that $b_{n,j}$ is an essential loop for both $j=1,2$; by passing to a subsequence if necessary.  Since $f_j$ is continuous, the paths $t_{n,j}$ and $f_{j*}(a_n)_r$ must both form null sequences.  We may choose an increasing sequence $i_{n}$ such that  $\diam{(t_{i_{n},j})}$ is decreasing and $\diam{(f_{j*}(a_n)_r)} < \dfrac{\diam{(t_{i_{n-1},j})}}2$ for all $n\geq i_{n}$.  The subsequence $i_n$ does not depend on $j$.


There exists an $r_{n,j}$ such that for some $A_{n,j}$ and $ B_{n,j}$,
$$w^{A_{n,j}}_{B_{n,j}}(T_{k_j}) < r_{n,j}\leq  w^{A_{n,j}}_{B_{n,j}}(b_{i_{n,j}}*\cdots* b_{i_{n,j}}),$$
the concatenation of $r_{n,j}$ copies of $b_{i_{n,j}}$.

Fix $\beta_{k_j}$ a $(k_1,k_2)$-reduced path such that $\pi_{k_j}\circ \beta_{k_{j}}=_p T_{k_j}$.  There exists an $r_{n,j}^{'}$ such that for some $A_{n,j}^{'}$ and $ B_{n,j}^{'}$ disjoint half planes with boundary parallel to $k_j$, we have
$$w^{A_{n,j}^{'}}_{B_{n,j}^{'}}(\pi_{k_j}\circ \beta_{k_{j^c}}) +4 < r_{n,j}'\leq  w^{A_{n,j}^{'}}_{B_{n,j}^{'}}(b_{i_{n,j}}*\cdots* b_{i_{n,j}})$$
the concatenation of $r_{n,j}'$ copies of $b_{i_{n,j}}$, where $j^c = \{1,2\}\backslash \{j\}$

Let $r_n = \max \{r_{n,1}, r_{n,2}, r_{n,1}', r_{n,2}'\}$.

We will inductively define integers $s_{n,j}$ and homotopy classes $c_n$ in $\mathbb H$.  Let $s_{1,j} = 0$ for $j = 1,2$ and  $c_1$ be the identity in $\mathbb H$.  Suppose $s_{n,j}$ and $c_i$ are defined for $1\leq i<n$ and $j = 1,2$.

Again, there exists an $s_{n,j}$ such that for some $A^{''}_{n,j}$ and $ B^{''}_{n,j}$, $$w^{A^{''}_{n,j}}_{B^{''}_{n,j}}(f_{j*}(\prod\limits_{i=1}^{n-1}(c_i))_r*t_{i_n,j}) < w^{A^{''}_{n,j}}_{B^{''}_{n,j}}(b_{i_n,j}*\cdots* b_{i_n,j}),$$
the concatenation of $s_{n,j}$ copies of $b_{i_n,j}$ for $j= 1,2$.  Let $s_n = \max \{s_{n,1}, s_{n,2}\}$ and $c_n = a_{i_n}^{r_n+s_n}$.  Let $a = \prod\limits_{n=1}^\infty c_n$.

Fix $h:I \to X$, a $(k_1,k_2)$-reduced path representative of $\phi(a)$. We can then apply the remainder of the proof of Lemma \ref{key} without change to find terminal supaths $\alpha_1, \alpha_2$ of $h$ mapping to $\overline{T}_{k_1}, \overline{T}_{k_2}$ respectively.

We may assume that $\alpha_1 =_p \gamma *\alpha_2$ for some subpath $\gamma$ of $h$. Suppose that $\pi_{k_2}\circ \gamma$ is non-degenerate.  By our construction of $a$,  $b_{i_n, 2}^{r_n}$ is a subpath of $\pi_{k_2}\circ \gamma$ for some $n$.

Since $\pi_{k_1}\circ \alpha_1 =_p \pi_{k_1}\circ \beta_{k_1}$, we can apply Lemma \ref{different delineations} and Lemma \ref{preserve0} to obtain $ w^{A_{n,j}^{'}}_{B_{n,j}^{'}}( \pi_{k_2}\circ \alpha_1) \leq w^{A_{n,j}^{'}}_{B_{n,j}^{'}}(\pi_{k_2}\circ \beta_{k_{1}}) +4 $.

Hence $w^{A_{n,j}^{'}}_{B_{n,j}^{'}}( \pi_{k_2}\circ \gamma)\leq w^{A_{n,j}^{'}}_{B_{n,j}^{'}}( \pi_{k_2}\circ \alpha_1) < r_n$ and $r_n \leq w^{A_{n,j}^{'}}_{B_{n,j}^{'}}(b_{i_n,2}^{r_n}) \leq w^{A_{n,j}^{'}}_{B_{n,j}^{'}}( \pi_{k_2}\circ \gamma)$ which is the desired contradiction.

Thus $\pi_{k_2}\circ \alpha_1 =_p T_{k_2}$ and $\pi_{k_1}\circ \alpha_1 =_p T_{k_1}$ as desired.

\end{proof}

We will now show that $\alpha$ is the path such that $\widehat\alpha
\circ\phi$ is induced by a continuous map.

\begin{lem}  Let $\alpha_k:I\to X$ be a path with the property that $\pi_k\circ\alpha_k = T_k$, up to reparametrization. Let $U$ be a $\pi_k$-saturated
neighborhood of $\pi_k^{-1}(T_k(1))$. For sufficiently large
$n$, every element of $\widehat\alpha\circ\phi(\mathbb H_n)$ has a representative in $U$.

\end{lem}

\begin{proof}
If $g$ is a loop based at a point $y\in\pi_k^{-1}(x)$ and
$w^{\pi_k^{-1}(x)}_{U^c}([g])=0$, then $g$ has a representative in $U$.

Let $U$ be a $\pi_k$-saturated neighborhood of $\pi_k^{-1}(T_k(1))$.
Let $U'$ by an open $\pi_k$-saturated neighborhood of
$\pi_k^{-1}(T_k(1))$ with closure contained in the interior of $U$.
We must show that, for some sufficiently large $n$,
$w^A_B(\widehat\alpha\circ\phi(b))=0$ for all $b\in \mathbb H_n$ where
$A=\pi_k^{-1}(T_k(1))$ and $B= \overline {U'}^c$.

Since $\widehat T_k\circ\pi_{k*}\circ\phi$ is induced by a
continuous map, there exists an $N$ such that, for all $n>N$,
$\widehat T_k\circ\pi_{k*}\circ\phi(H_n)
\subset\pi_1(\pi_k({U'}),T_k(1))$. Hence $w^{A_k}_{B_k}(\widehat
T_k\circ\pi_{k*}\circ\phi(b)) = 0$ for all $b\in \mathbb H_n$ where $n>N$.

For $b\in\mathbb H_n$, where $n>N$, let $f$ be a $k$-reduced representative
of $\widehat\alpha\circ\phi(b)$. Then the $w^A_B
(\widehat\alpha\circ\phi(b)) \leq w^A_B( f)  = w^{A_k}_{B_k}
(\pi_k\circ f) = w^{A_k}_{B_k}(\widehat T_k\circ\pi_{k*}\circ
\phi(b))= 0$.

\end{proof}

\begin{thm} \label{contplanar}Let $\phi:\mathbb H \to\pi_1(X,x_0)$ a
homomorphism into the fundamental group of a planar Peano continuum
$X$. Then there exists a continuous function $f:(\textbf{E},0) \to
(X,x)$ and a path $\alpha:(I,0,1)\to (X,x_0,x)$, such that $f_* =\widehat\alpha \circ \phi$.  Additionally, if the image of $\phi$ is uncountable then $\alpha$ is unique up to homotopy relative to its endpoints.
\end{thm}

\begin{proof}
For $k$ and $l$ nonparallel lines in the plane, there exists a path $\alpha$ in $X$ such that $\pi_k(\alpha) = T_k$ and $\pi_l(\alpha) = T_l$, by Lemma \ref{pmaximalheadtail}.

It is sufficient to show that for any neighborhood $U$ of
$\alpha(1)$ there exists an $N$ such that every element of $\widehat\alpha \circ \phi
(\mathbb H_n)$ has a representative in $U$.

This is done by finding $U_l$
and $U_k$ such that $U_k\cap U_l\subset U$ and $U_l$ is a
$\pi_l$-saturated neighborhood of $\pi_l^{-1}\circ\pi_l(\alpha(1))$
and $U_k$ is a $\pi_k$-saturated neighborhood of $\pi_k^{-1}\circ\pi_k(\alpha(1))$.  The uniqueness follows from the uniqueness of $T_k$ and the injectivity of $\pi_{k*}$.
\end{proof}

\section{Homomorphisms induced by continuous functions}\label{sec2}

Throughout this section, we will assume all paths are from the unit interval $[0,1]$.  We will use $j$ for the inclusion map of subsets of $X$ into $X$. When there is no chance of confusion, we will allow the domain of $j$ to change without comment.

\begin{defn}
Let $\phi: \pi_1(X,x_0)\to \pi_1(Y,y_0)$ be a homomorphism. We will say $x$ is $\phi$-bad if $\phi\circ\hat\alpha\circ j_*(\pi_1(B_\epsilon(x),x))$ is nontrivial for all $\epsilon>0$ where $\alpha$ is a path from $x_0$ to $x$.  We will say that $X$ is $\phi$-bad, if it is at each point. We will write $B_\phi(X)$ for the set of $\phi$-bad points of $X$.  If $X=Y$ and $\phi= \text{id}$, then this definition is the definition given in Section \ref{sec3} for the \emph{bad} set of $X$ and we will simply write $B(X)$.

Note this is independent of the chosen $\alpha$.
\end{defn}

Cannon and Conner in \cite {cc4} defined $B(X)$ for a certain class of subset of the plane which where homotopy equivalent to planar Peano continua.  It is a trivial exercise to see that these definitions are compatible.  Conner and Eda in \cite{ce} defined a similar set for one-dimensional spaces which they referred to as the set of \textit{wild} points.  While both terms have appeared in the literature, here we will use the term bad.

The following theorem, while not explicitly stated, follows from the proof of Theorem 1.2 in \cite{eda3}.

\begin{thm}[Eda, \cite{eda3}]\label{retractEda-thm}

Let $\phi:\pi_1(X, x_0) \to G$ be a homomorphism from the fundamental group of a one-dimensional Peano continuum $X$ to a group $G$ which is a subgroup of an inverse limit of free groups.  Suppose that $x_0 \in B_\phi(X)$.  Then there exists a retraction $r: (X,x_0) \to (X,x_0)$ such that $ B_\phi(X)  = B\bigl(r(X)\bigr)$ and $\phi([f]) = \phi([r\circ f])$ for all loops $f$ based at $x_0$.

\end{thm}

\begin{prop}\label{corr} Let $X$ and $Y$ be one-dimensional or planar Peano continua.  Suppose that $\phi : \pi_1(X, x_0) \to \pi_1(Y, y_0)$ is a homomorphism between their fundamental groups. Then for every path $\alpha: (I,0,1)\to (X,x_0,x)$ such that $x\in B_{\phi}(X)$, there exists a unique (up to homotopy rel endpoints) path $\psi(\alpha): (I,0,1)\to (Y,y_0,y)$  where $\psi(\alpha)$ depends only on the homotopy class (rel endpoints) of $\alpha$  and $y$ depends only on $x$ which enjoys the following property.

\begin{equation}\label{univ_property}\forall g: (\textbf E, 0) \to (X,\alpha(1)),\hspace{1mm} h:(\textbf E, 0) \to (Y,\psi(\alpha)(1))  \text{ such that } \phi\circ\hat{\overline{\alpha}}\circ g_* =\hat{\overline{\psi(\alpha)}}\circ h_*.\end{equation}
\end{prop}


If $X$ and $Y$ one-dimensional, then this is Lemma 5.1 in \cite{eda}.  The proof we present here is essentially that same as Eda's proof.  We give the proof for the planar case here to establish notation and to show the insignificant changes required in the planar situation.

\begin{proof}
Since $x\in B_\phi(X)$, there exists a sequence of loops, $\{\mathbf b_i\}$, based at $x$ which converge to $x$ with the property that $[\alpha*\mathbf b_i*\overline \alpha]$ (or the conjugation by any other path from $x_0$ to $x$) map non-trivially under $\phi$.  Let $g$ be a continuous  map which sends $\mathbf a_i$ to $\mathbf b_i$.  Theorem \ref{contplanar} implies that there exists a unique $\psi(\alpha):(I,0,1)\to (Y,y_0,y)$ and an $h:(\textbf{E},0)\to (Y,y)$ such that $\phi\circ\widehat{\overline\alpha}\circ g_* = \widehat{\overline {\psi(\alpha)}}\circ h_*$.

Suppose $g':( \textbf{E}, 0)\to (X,x)$ is any continuous function.  Let $g'':(\tb E, 0)\to (X,x)$ be the function which sends $\textbf{E}_e$ to $X$ by $g$ and sends $\textbf{E}_o$ to $X$ by $g'$.  Notice $g''$ is
continuous since $g(0)= g'(0) = x$.  Then there exists $\mf c$ and $h'':(\textbf{E},0)\to (Y,y)$ such at $\phi\circ\widehat{\overline\alpha}\circ g''_* = \widehat{\overline{\mf c}} \circ h''_*$.

Then $ \widehat{\overline {\psi(\alpha)}}\circ h_* = \phi\circ \widehat{\overline{\alpha}}\circ g_* = \phi\circ
\widehat{\overline{\alpha}}\circ (g''|_{\textbf{E}_e})_* = \widehat{\overline{\mf c}} \circ (h''|_{\textbf{E}_e})_*$. Then the
uniqueness in Theorem \ref{contplanar} implies that $\psi(\alpha)$ is homotopic to $\mf c$.  Hence $\phi\circ
\widehat{\overline{\alpha}}\circ g'_* = \phi\circ \widehat{\overline{\alpha}}\circ (g''|_{\textbf{E}_o})_* =
\widehat{\overline{\psi(\alpha)}} \circ (h''|_{\textbf{E}_o})_*$.

Now we must only show that $y$ depends only on $x$. Suppose $\alpha':(I,0,1)\to (X,x_0,x)$ is a path which is not necessarily
homotopic to $\alpha$.  Then for some path $\psi(\alpha')$ and $h'$, we have $\phi\circ\widehat{\overline\alpha'}\circ g_* = \widehat{\overline{\psi(\alpha')}} \circ h'_*$.

\begin{align*}
\widehat{\overline{\psi(\alpha')}} \circ h'_*(u) &= \phi\circ\widehat{\overline\alpha'}\circ g_*(u)\\
&= \phi\circ\hat{\overline{\alpha'*\overline{\alpha}}}\circ\hat{\overline {\alpha}}\circ g_*(u)\\
&= \phi(\alpha'*\overline{\alpha})\cdot\phi\circ\hat{\overline{\alpha}}\circ g_*(u)\cdot (\phi(\alpha'*\overline{\alpha}))^{-1} \\ &= \phi(\alpha'*\overline{\alpha})\cdot \widehat{\overline {\psi(\alpha)}}\circ h_*(u)\cdot \phi^{-1}(\alpha'*\overline{\alpha})
\end{align*}

Hence $h'(\mathbf{a}_n)$ is freely homotopic to $h(\mathbf{a}_n)$ for all $i$.

If $\psi(\alpha')(1)\neq \psi(\alpha)(1)$, then we can find $n$ sufficiently large such that $h(\mathrm{a}_n)$ and $ h'(\mathrm{a}_n)$ can be separated by disjoint balls in $Y$. This contradictions Lemma \ref{homotoped off}.  (We can choose $n$
sufficiently large such that the disjoint balls satisfy the hypothesis of Lemma \ref{homotoped off}.)
\end{proof}

\begin{thm}[Meistrup, \cite{precm}]\label{arcreduced}
Every one-dimensional Peano continuum is homotopy equivalent to a one-dimensional Peano continuum $X$ such that $X$ is a finite connected graph or $X\backslash B(X)$ is null sequence of disjoint open arcs with end points in $B(X)$.
\end{thm}

Such a continuum is called \emph{arc-reduced}.

We are now ready to prove:

\begin{thm}\label{cont1dim}
Let $\phi: \pi_1(X,x_0)\to \pi_1(Y,y_0)$ be a homomorphism where $X$ is a
one-dimensional Peano continuum and $Y$ is a planar Peano continua.  Then there exists a path $T: (I,0,1)\to (Y,y_0,y)$ and a
continuous function $f:X\to Y$ such that $\hat T\circ\phi = f_*$.
\end{thm}

When both $X$ and $Y$ are one-dimensional, this is Theorem 1.2 in Eda's recent paper \cite{eda3}.  We will show how to adapt Eda's to the case where $Y$ is a planar Peano continuum.

\subsection{Proof of  Theorem \ref{cont1dim}}

Let $\phi: \pi_1(X,x_0)\to \pi_1(Y,y_0)$ be a homomorphism where $X$ is a one-dimensional Peano continuum and $Y$ is a planar Peano continua.

\textbf{Observation:} Let $r:X\to X$ be any retraction of $X$ such that $\ker{r_*}\subset \ker{\phi}$.  Then there exists a unique homorophism $\phi': \pi_1\bigl(r(X), r( x_0)\bigr)\to \pi_1(Y, y_0)$ such that $\phi\bigl([\alpha]\bigr) = \phi'\bigl([r\circ\alpha]\bigr)$ for every loop $\alpha$ in $X$.  In addition, if there exists a path $T: (I,0,1)\to (Y,y_0,y)$ and a continuous function $f':r(X)\to Y$ such that $\hat T\circ\phi' = f'_*$; then $$\hat T\circ\phi = \hat T\circ(\phi'\circ r_*) = f'_*\circ r_* = (f'\circ r)_*.$$

Thus by Theorem \ref{retractEda-thm} and  Theorem \ref{arcreduced}, we may assume that $B(X) = B_\phi(X)$ and $X$ is arc-reduced. Since we are only concerned about $\phi$ up to composition with a change of base point isomorphism, we may assume that $x_0\in B(X)$.

Fix $k,l$ non-parallel lines in the plane.  Then there exists paths $T_k,T_l$ in $\pi_k(Y), \pi_l(Y)$ respectively and functions $f_k: X \to \pi_k(Y)$,$f_l: X \to \pi_l(Y)$   such that $\hat T_k\circ\pi_{k*}\circ\phi = f_{k*}$ and $\hat T_l\circ\pi_{l*}\circ\phi = f_{l*}$

Since $X$ is arc-reduced, $X\backslash B(X)$ is a null sequence of disjoint open arcs, $A_i$.  For each i, fix a homeomorphism $m_i: (0,1)\to A_i$. Then extend $m_i$ to $[0,1]$ by sending $0,1$ to the corresponding endpoints of $A_i$ in $B(X)$.  Without loss of generality, we may assume that $f_k, f_l$ take each $m_i$ to a reduced path in $\pi_k(Y), \pi_l(Y)$ respectively.

To simplify notation; we will use $[f]_r$ to represent the reduced representative of $[f]$ when $f$ is a path in a one-dimensional continuum and $[f]_{k,l}$ to represent a $(k,l)$-reduced representative of $f$ when $f$ is a path a planar Peano continuum.

Proposition \ref{corr} gives us a function $\psi$ from the set of reduced paths with initial point $x_0$ and terminal point in $B(X)$ to the set of paths in $Y$ from $y_0$ to points in $B(Y)$.  We want to extended $\psi$ to reduced paths from $x_0$ to
arbitrary points of $X$.

Let $\mf m_i$ be a $(k,l)$-reduced representative of $[\overline{\psi(\beta_i)} * \psi(\beta_i*m_i)]$ where $\beta_i$ is a path from $x_0$ to $m_i(0)$.  The next two lemmas will show that $\mf m_i$ is independent of our chosen $\beta_i$.

The proofs we present here for the next two lemmas is essentially that same as Eda's proofs of the corresponding lemmas.

\begin{lem}\label{eda3.3}[Lemma 3.3, \cite{eda3}]
Let $\alpha, \beta $ be paths from $x_0$ to $x\in B(X)$.  Then $\phi(\alpha*\overline\beta)= [\psi(\alpha)*\overline{\psi(\beta)}]$.

\end{lem}

\begin{proof}
For all continuous $h: (\tb E,0)\to (X,x)$, there exists $g,g': (\tb E,0)\to (Y,y)$ such that $\phi\circ\hat{\overline \beta}\circ h_* = \hat{\overline{\psi(\beta)}}\circ g_*$ and $\phi\circ\hat{\overline \alpha}\circ h_* = \hat{\overline{\psi(\alpha)}}\circ g'_*$. Then for all $u\in\mb H$,
\begin{align*}
\hat{\overline{\psi(\alpha)}} \circ g'_*(u) &= \phi\circ\hat{\overline {\alpha}}\circ h_*(u)\\
&= \phi\circ\hat{\overline{\alpha*\overline{\beta}}}\circ\hat{\overline {\beta}}\circ h_*(u)\\
&= \phi\bigl([\alpha*\overline{\beta}]\bigr) \cdot \phi\circ\hat{\overline{\beta}}\circ h_*(u)\cdot \bigl(\phi\bigl([\alpha*\overline{\beta}]\bigr)\bigr)^{-1} \\
&= \phi(\alpha*\overline{\beta})\cdot \hat{\overline{\psi(\beta)}}\circ g_*(u)\cdot \bigl(\phi(\alpha*\overline{\beta})\bigr)^{-1} \\ &=\hat{\overline{\phi(\alpha*\overline\beta)_{k,l}*\psi(\beta)}}\circ g_*(u)
\end{align*}

Then Theorem \ref{contplanar} implies that $\psi(\alpha)$ is homotopic to $\phi(\alpha*\overline\beta)_{k,l}*\psi(\beta)$ which completes the proof.
\end{proof}

\begin{lem}\label{eda3.4}[Lemma 3.4, \cite{eda3}]
The path $\mf m_i$ is independent (up to homotopy rel endpoints) of the path $\beta_i$ from $x_0$ to the initial point of $m_i$.
\end{lem}

\begin{proof}
Let $\beta$, $\gamma$ be paths from $x_0$ to the initial point of
$m_i$.  Then by Lemma \ref{eda3.3},
\begin{align*}
\Bigl(\overline{\psi(\gamma)}* \psi(\gamma*m_i) * \overline{\psi(\beta*m_i)} * \psi(\beta)\Bigr)& =_{0,1}
\overline{\psi(\gamma)}*\Bigl( \psi(\gamma*m_i) * \overline{\psi(\beta*m_i)} \Bigr)* \psi(\beta)\\
& =_{0,1} \overline{\psi(\gamma)} * \phi([\gamma*m_i * \overline{m_i}*\overline{\beta}])_{k,l}  *\psi(\beta) \\
&=_{0,1} \overline{\psi(\gamma)} * \phi([\gamma*\overline{\beta}])_{k,l} \cdot \psi(\beta) \\
&=_{0,1} \overline{\psi(\gamma)}*\psi(\gamma)*\psi(\overline{\beta}) * \psi(\beta) =_{0,1} x_0
\end{align*}

where $x_0$ is the constant path.

\end{proof}

We can now extended $\psi$ to reduced paths from $x_0$ to arbitrary points of $X$.  Let $\alpha:(I,0,1)\to (X, x_0, x)$ be a reduced path.  If $x\in B(X)$, then $\psi(\alpha)$ is already defined.  Otherwise, $x = m_i(t_0)$ for some $i$ and $t_0$.  Then $\alpha$, up to re-parametrization, is $\alpha'*m_i|_{[0,t_0]}$ or $\alpha'*\overline{m_i}|_{[0,1-t_0]}$ where $\alpha'$ is uniquely determined and is a path between points in $B(X)$.  Then we will define $\psi(\alpha)$ to be $\psi(\alpha')*\mf m_i|_{[0,t_0]}$ if $\alpha =\alpha'*m_i|_{[0,t_0]}$ or $\psi(\alpha')*\overline{\mf m_i}|_{[0,1-t_0]}$ if $\alpha = \alpha'*\overline{m_i}|_{[0,1-t_0]}$.

We are now ready to define our function between $X$ and $Y$ which will induce $\phi$ up to conjugation.  Let $f:X \to Y$ by $f(x)= \psi(\alpha)(1)$ for a path $\alpha$ from $x_0$ to $x$.  If $x\in B(X)$, this is independent of $\alpha$ by Proposition \ref{corr}.  If $x\in X\backslash B(X)$, this is independent of $\alpha$ by Lemma \ref{eda3.4}.  By construction, $f|_{m_i}$ is continuous.

\begin{lem}\label{composition}
Let $x\in B(X)$ and $\alpha: (I,0,1) \to (X,x_0,x)$. Then $\pi_k\circ \psi(\alpha) =_{0,1} T* f_k\circ \alpha$ and $\pi_k\circ f(x) = f_k(x)$.
\end{lem}

\begin{proof}
By construction $\psi(\alpha)$ is the unique path satisfying (\ref{univ_property}) for the homomorphism $\phi$ .  By construction $f(x)= \psi(\alpha)(1)$.

Fix $g: (\textbf E, 0) \to \bigl(X,\alpha(1)\bigr)$.

By Proposition \ref{corr}, there exists $h:(\textbf E, 0) \to \bigl(Y,\psi(\alpha)(1)\bigr)$ such that $\phi\circ\hat{\overline{\alpha}}\circ g_* =\hat{\overline{\psi(\alpha)}}\circ h_*$.  Hence

\begin{align*}
\pi_{k*}\circ\phi\circ\hat{\overline{\alpha}}\circ g_*
& = \pi_{k*}\circ\hat{\overline{\psi(\alpha)}}\circ h_{*}\\
& = \hat{\overline{\pi_{k}\circ \psi(\alpha)}}\circ \pi_{k*}\circ h_{*}\\
& = \hat{\overline{\pi_{k}\circ \psi(\alpha)}}\circ \bigl(\pi_{k}\circ h\bigr)_*
\end{align*}

and, as well,

\begin{align*}
\pi_{k*}\circ\phi\circ\hat{\overline{\alpha}}\circ g_*
& = \hat{\overline{T}}\circ f_{k*}\circ \hat{\overline{\alpha}}\circ g_*\\
& = \hat{\overline{T}}\circ \hat{\overline{f_k\circ \alpha}}\circ \bigl(f_k\circ g\bigr)_* \\
& =  \hat{\overline{T*f_k\circ \alpha}} \circ \bigl(f_k\circ g\bigr)_*.
\end{align*}

The uniqueness conditions from Proposition \ref{corr} imply that $f_k(x) = f_k\circ \alpha(1) = \pi_{k}\circ \psi(\alpha)(1)= \pi_k\circ f(x)$ and $\pi_k\circ \psi(\alpha) =_{0,1} T* f_k\circ \alpha$.
\end{proof}

\begin{lem}\label{almost composition}
$\pi_k\circ f\circ m_i = _{0,1} f_k\circ m_i$
\end{lem}

\begin{proof}
By Lemma \ref{composition},
\begin{align*}
\pi_k\circ f \circ m_i = \pi_k \circ\mf m_i & = _{0,1} \overline{\pi_k \circ \psi(\beta_i)}* \pi_k \circ \psi(\beta_i* m_i) \\
& =_{0,1} \bigl(\overline{T*f_k\circ \beta_i}\bigr) * \bigl(T* (f_k\circ \beta_i)* (f_k \circ m_i)\bigr)\\
& =_{0,1} f_k\circ m_i
\end{align*}
\end{proof}

Note that both Lemma \ref{composition} and Lemma \ref{almost composition} hold when $k$ is replaced by $l$.

\begin{lem}
$f$ is continuous.

\end{lem}

\begin{proof}
By Lemma \ref{continuous maps} and \ref{composition}, we know that $\pi_k\circ f$ and $\pi_l\circ f$ are continuous of $B(X)$.
Since $m_i(I)$ from a null sequence in $X$, $f_k\bigl(m_i(I)\bigr)$ and $f_l\bigl(m_i(I)\bigr)$ both form null sequence of sets in $Y_k,Y_l$ respectively.  Since $\mf m_i$ is $(k,l)$-reduced and $f_k\circ m_i$ is reduced, $\pi_k\bigl(f\bigl(m_i(I)\bigl)\bigl)) = f_k\bigl(m_i(I)\bigr)$ and $\pi_l\bigl(f\bigl(m_i(I)\bigl)\bigl) = f_l\bigl(m_i(I)\bigr)$. Then Lemma \ref{continuous maps-corrollary} implies that $f\bigl(m_i(I)\bigr)$ forms a null sequence of sets in $Y$.  Thus $f$ is continuous on $\bigcup_i m_i(I)$ which completes the proof.
\end{proof}

\begin{lem}
$\pi_{k*}\circ f_* = f_{k*}$
\end{lem}

\begin{proof}
Suppose that $\alpha: (I,0,1) \to (X, x_0, x_0)$.   Choose open subinterval $I_j$ of $I$ with the canonical ordering such that $\cup I_j = I\backslash \alpha^{-1}\bigl(B(X)\bigr)$.  Then $\pi_{k}\circ f\circ \alpha|_{\overline I_j}$ is a reparametrization of  $f_{k}\circ \alpha|_{\overline I_j}$ and $\pi_{k}\circ f = f_{k}$ on $B(X)$ which completes the proof.
\end{proof}

All that remains to show is that $\phi = \hat{\overline {\alpha}}\circ f_*$ for some path $\alpha$ from $y_0$ to $f(x_0)$.  As in Proposition \ref{pmaximalheadtail}, we can find a path $\alpha$ such that $\pi_k(\alpha)= _{0,1} T_k$.  Then

\begin{align*}
\pi_{k*}\circ \phi &= \hat{\overline{T_k}}\circ f_{k*}\\
& = \hat{\overline{T_k}}\circ \pi_{k*}\circ f_* \\
& = \pi_{k*} \circ \hat {\overline{\alpha}} \circ f_*.
\end{align*}

Since $\pi_{k*}$ is injective $\phi = \hat {\overline{\alpha}} \circ f_*$ and the uniqueness of $\alpha$ follows from the uniqueness of $T_k$ which completes the proof of Theorem \ref{contplanar}.

\section{Recovering the bad set}\label{sec3}

Recall that Eda and Conner in \cite{ce} showed that the fundamental group of a one-dimensional continuum which is not semilocally simply connected at any point can be used to reconstruct the original space.  The construction that was used there is similar to the one we will present here for the planar case.

\begin{defn}
We will denote the set of points at which $X$ is not semilocally simply connected by $B(X)$, the \emph{bad set of $X$}.
\end{defn}

\begin{lem} \label{fixed}Let be $X$ a planar Peano continuum.  Then $B(X)$ is fixed by every self-homotopy of $X$.\end{lem}

\begin{proof}

Let $f:X\to X$ be homotopic to the identity on $X$.  Then any loop
is homotopic to its image under $f$.  Suppose that $x\in B(X)$ such
that $f(x) \neq x$. Since $X$ is not semilocally simple connected
at $x$, there exists a closed disk neighborhood $D$ of $x$ in the
plane with the property that $\partial D$ is not contained in $X$
and $D\cap f(D) = \emptyset$. Then any loop contained in $D$ can be
homotoped into $f(D)$.  However; this is a contradiction of
Lemma \ref{homotoped off}, since there exist essential loops
contained in the interior of $D\cap X$.

\end{proof}

\begin{defn}We will say that $H$ is a \emph{pseudo-Hawaiian earring
subgroup of $G$}, if $H$ is a subgroup of $G$ and the uncountable homomorphic image of $\mb H$.\end{defn}

It will sometimes be convenient to talk about a generating set, in the sense of infinite products, for a (pseudo-)Hawaiian earring subgroup of $\pi_1(X,x_0)$.  For each (pseudo-)Hawaiian earring subgroup $H\in\pi_1(X,x_0)$, we may choose an (homomorphism) isomorphism $\phi: \mathbb H
\to H$.  We will say that $\{\phi(a_i)\}$ \emph{generates (in the
sense of infinite products)} $H$ and will write $H =
\infgen{\phi(a_n)}$.

In this section, we will show that there exists a natural equivalence
relation on pseudo-Hawaiian earring subgroups.  Natural in the sense that
it is determined by the topology and determines the topology of $B(X)$.

A natural question is what are the possible isomorphism types of
pseudo-Hawaiian earring subgroups.  Conner has conjectured that every
uncountable homomorphic image of an Hawaiian earring group which embeds in
the fundamental group of a planar or one-dimensional Peano continuum is a
Hawaiian earring group.

\subsection{Equivalence classes of pseudo-Hawaiian earring subgroups}

We are now ready to prove:

\textbf{Theorem \ref{bad}} \emph{The homeomorphism type of $B(X)$ is completely determined by $\pi_1(X,x_0)$, where $B(X)$ is the subspace of points at which $X$ is not semilocally simply connected.}

\begin{defn}
We will say that $x$ is a \emph{wedge point} for $H\subset\pi_1(X,x_0)$ a pseudo-Hawaiian earring subgroup if, there exists a surjective homomorphism $\psi: \mb H \to H$ with the property that $\psi = \hat\alpha\circ f$ for a continuous function $f:\mathbb{E} \to X$ such that $f(0) = x$.
\end{defn}

The following two lemmas are immediate from the definitions.

\begin{lem}\label{sur}
If $X$ is a planar Peano continuum and $x\in B(X)$, then there
exists $H$ a pseudo-Hawaiian earring subgroup of $\pi_1(X,x_0)$ with
$x$ as its wedge point.
\end{lem}

\begin{lem}\label{conjugatewedgepoint}
If $x$ is a wedge point for $H$ a pseudo-Hawaiian earring subgroup, then $x$ is a wedge point for $\widehat\alpha (H)$.

\end{lem}

We should show the existence and uniqueness of wedge points.

\begin{lem}
There exists a unique wedge point for each pseudo-Hawaiian earring
subgroup.
\end{lem}

\begin{proof}
The existence follows Theorem \ref{contplanar}.

Suppose the a pseudo-Hawaiian earring subgroup $H$ has two distinct
wedge points $x$ and $x'$ which correspond to two homomorphism $\psi$ and $\psi'$ with image $H$.  Then there exists $f:(\textbf{E},0)
\to (X,x)$, $f:(\textbf{E},0) \to (X,x')$, $\alpha:(I,1)\to
(X,x)$, and $\alpha':(I,1)\to (X,x')$ such that
$\widehat\alpha\circ f_* = \psi$ and $\psi' = \widehat\alpha'\circ f_*' $.

Let $U$ and $U'$ be disjoint disks in the plane whose boundaries are not
contained in $X$ such that $U\cap X$ and $U'\cap X $ are neighborhoods of
$x$ and $x'$ respectively. Since $f$ and$f'$ are continuous, the image of
$f(\mathrm{a}_i)$ is eventually contained in $U\cap X$ and the image of
$f'(\mathrm{a}_i)$ is eventually contained in $U'\cap X$.  Hence there
exists $N$ such that $f_*(\mb H_N)$ is contained in $U\cap X$ and
$f'_*(\mb H_N)$ is contained in $U'\cap X$

A countability argument shows that $\psi(\mb H_n)\cap\psi'(\mb H_n)$ is
nontrivial for all $n$. However, by the previous paragraph any element in
the intersection for $n>N$ can then be freely homotoped from $U$ to $U'$ which contradicts Lemma \ref{homotoped off}.
\end{proof}

This lemma also implies that the definition of a wedge point is independent of our choice of homomorphism.

We will let $\hes$ be the set of pseudo-Hawaiian earring subgroups of
$\pi_1(X,x_0)$.  Then we can now use the previous lemmas to define
$\Upsilon:\hes\to B(X)$, by taking a pseudo-Hawaiian earring subgroup to
its wedge point. This function is surjective by Lemma \ref{sur}. However, Lemma
\ref{conjugatewedgepoint} implies that it is highly non-injective.  We
will now define an equivalence relation on $\hes$ to make $\Upsilon$
injective on equivalency classes.

\begin{defn}Let $H_1, H_2$ be pseudo-Hawaiian earring subgroups of
$\pi_1(X,x_0)$. Then we will say that $H_1$ is \emph{similar} to $H_2$ if
there exists a $g\in\pi_1(X,x_0)$ and an $H\in \pi_1(X,x)$ a pseudo-Hawaiian earring subgroup such
that $H_1,gH_2g^{-1}\subset H$.  We will write $H_1\sim H_2$.
\end{defn}
If $X$ is one-dimensional these equivalence classes are the same as
maximal compatible nonempty subfamilies of the set of subgroups of
$\pi_1(X,x_0)$ in \cite{ce}.

The following is well know.

\begin{thm}\label{art}
For $\psi:\mathbb H \to \mathbb F$ a homomorphism where $\mathbb
F$ is a free group,  there exists an $i$ such that $\psi$
factors through $P_{i*}$.
\end{thm}

If we consider homomorphisms from the natural inverse limit containing
$\mb H$ to free groups, then this is a theorem of Higman \cite{hig}.
Morgan and Morrison also gave a topological proof in \cite{mm}. When we
consider homomorphism from $\mb H$, this is a result of Theorem 4.4 in
\cite{cc3} and a proof can be found in \cite{cs}.

\begin{lem}\label{ifwedge}
If $H_1$ and $H_2$ have the same wedge point, then $H_1\sim H_2$.
\end{lem}

\begin{proof}
Let $x$ be a wedge point for $H_1$ and $H_2$.  Let $\psi_i: \mb H \to \pi_1(X,x_0)$ be a homomorphism with image $H_i$ for $i\in\{1,2\}$.  Then there exists
continuous functions $f_i:(\textbf{E}, 0)\to (X,x)$ and
$\alpha_i:(I,1)\to (X,x)$ such that $\widehat\alpha_i\circ
f_{i*} = \psi_{i}$.

Let $g: (\textbf E,0) \to (X,x)$ be sending $\textbf E_o$ to $f_1(\textbf E)$ and $\textbf{E}_e$ to $f_2(\textbf E)$, i.e. $\mathrm{a}_i$ goes to $f_1(\mathrm{a}_{\frac{i+1}2})$ for odd $i$ and $f_2(\mathrm{a}_\frac i2)$ for even $i$.

Hence  $\widehat\alpha_i(H_i)\subset f_*(\mb H)$  and $f_*(\mb
H)$ is a pseudo-Hawaiian earring subgroup. Then after conjugating $H_2$ by $\alpha_1*\overline{\alpha_2}$ and $f_*(\mb H)$ by $\alpha_1$, we have necessary inclusion to show that $H_1\sim H_2$.
\end{proof}

\begin{lem}\label{onlyifwedge}
If $H_1\sim H_2$ , then $H_1$ and $H_2$ have the same wedge point.
\end{lem}

\begin{proof}

It is sufficient to show that if $H_1$ and $H_2$ are pseudo-Hawaiian earring subgroups with the property that $H_1\subset H_2$, then $H_1$ and $H_2$ have the same wedge point.  There exists
continuous functions $f_i:(\textbf{E}, 0)\to (X,x_i)$ and paths $\alpha_i:(I,0,1)\to (X,x_0,x_i)$ such that $\widehat{\overline\alpha}_i\circ f_i(\mathbb H) = H_i$.

Suppose that $f_1(0) = x_1 \neq x_2 = f_2(0)$. By Lemma \ref{alem2}, we may find $r_i$ such that $B_{2r_1}^X(x_1) \cap B_{2r_2}^X(x_2) = \emptyset$.   Then we may choose $N$ such that $f_i(\textbf{E}_N) \subset B_{r_i}^X(x_i)$ for $i= 1,2$.

Let $X' = X \cup  B_{r_2}^{\mathbb R}(x_2)$.  Let $f_i':(\textbf{E}, 0)\to (X',x_i)$ denote the map obtained by extending the range of $f_i$ and $H_i' = \widehat{\overline\alpha}_i\circ f_i'(\mathbb H)$.   Then $H_2'$ is countable which implies that $H_1'$ is countable.  Hence, $ f_{2*}(a_n)$ is eventually trivial.

Fix $n> N$ such that $f_{2}'(\mathbf a_{n})$ is null-homotopic.  Choose a continuous map $h : \mathbb D \to X'$ such that $h|_{\partial \mathbb D} = f_2' (\mathbf a_{n})$.   Hence by Lemma \ref{cutoff}, $h$ can be modified to a map with image contained in $B_{2r_2}^{X'}(x_2)$.  Since $B_{2r_1}^X(x_1)$ and $ B_{2r_2}^X(x_2)$ are disjoint, $h$ maps into $X$ and $f_2(\mathbf a_{n})$ is null-homotopic.  This holds for any $n>N$; thus $f_{2*}(\mathbb H)$ is countable which contradicts our choice that $H_2$ was a pseudo-Hawaiian earring group.

Therefore $x_1 = x_2$.

\end{proof}

Clearly, $\sim$ is reflexive and symmetic. Then Lemmas
\ref{ifwedge} and \ref{onlyifwedge} show that $\sim$ is
transitive. Hence, $\sim$ is actually an equivalency relation.
We will use $[H]$ to denote the equivalency class of pseudo-Hawaiian
earring subgroups generated by the equivalency relation $\sim$.

We can now give an equivalent definition for the wedge point of a psuedo-Hawaiian earring subgroup.  Let $H = \infgen{g_n}$ be a psuedo-Hawaiian earring subgroup.  Then the wedge point of $H$ is the unique point, $x$, which enjoys the following property:
\emph{For every $\epsilon >0$, there exists an $N$ such that the set of homotopy classes  $\{g_n\}_{n>N}$ have representatives contained in $B_\epsilon(x)$}.

\begin{prop}\label{bijective}{$B(X)$ corresponds bijectively to equivalence classes of pseudo-Hawaiian earring subgroups of $\pi_1(X,x_0)$, i.e.
$\Upsilon^{-1} (x) = [H]$}

\end{prop}

Then it makes sense to talk about the \emph{wedge point of a
pseudo-Hawaiian earring subgroup equivalency class} being the wedge
point of any of its members.

\begin{defn}
Let $\{[H_n]\}$ be a sequence of equivalency classes of pseudo-Hawaiian earring subgroups of the fundamental group of a planar Peano continuum. We will say that $[H_n]$ \emph{converges} to $[H]$, a pseudo-Hawaiian earring subgroup equivalence class, if there exists a choice of representatives $\infgen{g_{n,i}}_i\in[H_n]$ and a sequence of natural numbers $\{M_n\}$ so that for any sequence
$\{k_n\}$, with $k_n> M_n$ the pseudo-Hawaiian earring subgroup
$\infgen{g_{n,k_n}}_n \in [H]$.  We will use all the standard terminology when referring to sequence of pseudo-Hawaiian earring subgroup equivalence classes.
\end{defn}

\emph{Actually since we have to be able to choose the representative for
the equivalence class, we may as well choose the representative such that $M_n = 1$.  This can be done by removing the  first finitely many big loops which will not change the equivalence class.}

Initially there is no reason to suppose that a convergent sequence has a unique limit point.  However, this will be a consequence of Proposition \ref{limitpoint} and Lemma \ref{ifwedge}.

\begin{prop}\label{limitpoint}
If $[H_n] \to [H]$, then $x_n \to x$ where $x_n$ is the wedge
point of $[H_n]$ and $x$ is the wedge point of $[H]$.
\end{prop}

\begin{proof}
Suppose that $[H_n]$ \emph{converges} to $[H]$ a pseudo-Hawaiian
earring subgroup equivalence class.  Then there exists a choice
of representatives $\infgen{g_{n,i}}_i\in[H_n]$ and $\{M_n\}$
such that if $k_n\geq M_n$, then $\infgen {g_{n,k_n}}_n\in[H]$.

We will proceed by way of contradiction.  Let $x_n$ be be the
wedge point for $[H_n]$.  Suppose that $x_n \not\to x$. Then
there is a subsequence $x_{n_i}$ with the property that $x_{n_i}
\to x' \neq x$.  Fix $A, A'$ disjoint open disk neighborhoods of
$x$ and $x'$ respectively, whose boundaries are not contained in
$X$. Then we may choose $\{M'_{n_i}\}_i$ and $N'$ such that
$g_{n_i,k_i}$ can be freely homotoped into $A'$ for $n_i>N'$ and
$k_i>M'_{n_i}$.  Let $M'_k = 0$ if $k \neq n_i$ for any $i$.

Fix $g_{n,k(n)}$ such that $k(n)> \max\{M'_{n}, \ M_{n}\}$.  Then
$\infgen {g_{n,k(n)}}_n\in[H]$ and there exists an $N$ such that,
for $n> N$, $g_{n,k(n)}$ can be freely homotoped into $A$.  Then for $n_i
>\max\{N',N\}$ $g_{n_i,k(n_i)}$ can be freely  homotoped into $A$ and $A'$, a
contradiction.

\end{proof}

The uniqueness of the limit of a convergent sequence then follows
by considering the corresponding sequence of wedge points.

Let $\overline{\mathcal{H}}$ be the set of equivalency classes of
pseudo-Hawaiian earring subgroups.  We will now construct a topology on
$\overline{\mathcal{H}}$.

\begin{defn}We will say that $A\subset\overline{\mathcal{H}}$ is \emph{closed} if it contains all of its limit points.  Then a
set is \emph{open} in $\overline{\mathcal H}$ if its complement is
closed.
\end{defn}

It is an easy exercise to show that this defines a topology on
$\overline{\mathcal H}$.

Let $\overline \Upsilon:\overline{\hes}\to
B(X)$ by $\overline\Upsilon([H]) = \Upsilon(H)$.  This is well defined by Lemma \ref{onlyifwedge} and a bijection by Proposition \ref{bijective}.

Proposition \ref{limitpoint} is exactly what is required for $\overline\Upsilon$ to be continuous.  Suppose $C\subset B(X)$ which is closed.  Consider $\{[H_n]\}\subset \overline\Upsilon^{-1}(C)$ which converges to $[H]$.  If $x_n$ is the wedge point of $[H_n]$, then $x_n$ is in $C$  and $x_n$ must converge to $x$ the wedge point of $[H]$.  Thus $x\in C$ and $[H]\in\overline\Upsilon^{-1}(C)$.

The next lemma will give the continuity of $\overline\Upsilon^{-1}$ immediately since $X$ is metric.

\begin{lem}
If $x_n\in B(X)$ and and $x_n\to x$, then the corresponding
pseudo-Hawaiian Earring subgroup equivalency classes converge.
\end{lem}

\begin{proof}
Let $H_n$ be a pseudo-Hawaiian Earring subgroup of $\pi_1(X,x_0)$ with wedge point $x_n$.  Then there exists a continuous function
$f_n: (\textbf E, 0)\to (X, x_n)$ which induces an injective
homomorphism.  There exists a path $T:I\to X$ with the property
that $T(0)= x$, $T(1)= x_1$, and $T(\frac1n)= x_n$.  Let
$\alpha:I\to X$ be any path from $x$ to $x_0$.  Let $\beta_n$ be
a parametrization of the path $\alpha* T\bigr|_{[0,\frac1n]}$.
Then $\infgen{\overline{\beta_n}*f_n(\mathrm{a}_i)*\beta_n}_i$ has wedge
point $x_n$.  Let $g_{n,i} =
\overline{\beta_n}*f_n(\mathrm{a}_i)*\beta_n$.  There exists $M_n$ such
that $f_n(\mathrm{a}_i)\subset B_\frac1n(x_n)$ for $i> M_N$.  If $k\geq
M_n$ for each $n$, then $\infgen {g_{n,k}}_n\in[H]$.
\end{proof}

Therefore we have know proved the following result.

\begin{thm}\label{bad}
The function $\overline \Upsilon:\overline{\hes}\to B(X)$ is a
homeomorphism, i.e. the homeomorphism type of $B(X)$ is completely determined by $\pi_1(X,x_0)$.
\end{thm}

\renewcommand{\thethm}{A\arabic{thm}}
 \setcounter{thm}{0}
\section*{APPENDIX}

\begin{lem}
Every non-degenerate null-homotopic loop in a one-dimensional space contains a proper non-degenerate subloop which is also null-homotopic.
\end{lem}

\begin{proof}
Suppose that $f: S^1 \to X$ is a non-degenerate null-homotopic loop in a one-dimensional space $X$.  Then $f$ factors through a map $f': S^1 \to D$ where $D$ is a one-dimensional planar dendrite (see \cite[Theorem 3.7]{cc3}).  Since $f(S^1)$ is non-degenerate, so is $D$ and there exists $x\in D$ such that $D\backslash\{x\}$ has at least two components.  Hence $S^{-1}\backslash \{f^{'-1}(x)\}$ has at least two components.  Let $A$ be the closure of a component of $S^{-1}\backslash \{f^{'-1}(x)\}$.  Then $f|_A$ is a proper null-homotopic subpath of $f$.  Since $f'|_A$ maps onto the closure of a component of $D\backslash\{x\}$ and the components of $D\backslash\{x\}$ are non-degenerate, $f|_A$ is non-degenerate.
\end{proof}

\begin{lem}
Suppose that $f:[0,3]\to X$ is a path into a one-dimensional space $X$ such that $f|_{[i,i+2]}$ is a reduced path for $i\in 0,1$.  If $f_{[1,2]}$ is a non-degenerate then $f$ is reduced.
\end{lem}

\begin{proof}

Suppose $f|_{[s,t]}$ is a non-degenerate null-homotopic subloop of $f$.  Since $f|_{[i,i+2]}$ is a reduced path for $i\in 0,1$; $s\leq 1<2\leq t$.

If $s_1\leq s_2\leq 1<2\leq t_1\leq t_2$ and $f|_{[s_i,t_i]}$ is a non-degenerate null-homotopic subloop of $f$, then $f|_{[s_2,t_1]}$ is a non-degenerate null-homotopic subloop of $f$.

Suppose that $s_1\leq s_2\leq \cdots \leq 1<2\leq \cdots \leq t_2 \leq t_1$.  Let $s = \lim s_i$ and $t = \lim t_i$. It is not hard to see that $f|_{[s,t]}$ is also a non-degenerate null-homotopic subloop of $f$ (for details see the proof of Theorem 3.9 in \cite{cc3}).

Then there exists a unique minimal interval $[s_0, t_0]$ such that $f|_{[s_0,t_0]}$ is a non-degenerate null-homotopic subloop of $f$.  But then $f$ contains no a proper non-degenerate null-homotopic subloops which contradicts the previous lemma.

\end{proof}

Greg Conner and Mark Meilstrup proved the following lemma in \cite{precm}.

\begin{lem} {\label{aalem1} Let $H$ be a function from the first-countable space
$X\times Y$ into $Z$.  Let $\{C_i\}$ be a null sequence of closed
sets whose union is $X$.  Suppose that $\{D_i= H(C_i\times Y)\}$ is
a null sequence of sets in $Z$ and $H$ is continuous on each
$C_i\times Y$.  If for every subsequence $C_{i_k}\to x_0$ there
exists a $z_0\in Z$ such that $D_{i_k} \to z_0$ then $H$ is
continuous on $X\times Y$.}\end{lem}

\begin{proof}
Consider a sequence $(x_n, y_n) \to (x_0,y_0)$.  For each $n$,
choose an $i_n$ such that $x_n\in C_{i_N}$.  If $\{C_{i_n}\}$ is
finite then by restricting $H$ to $\cup_n C_{i_n}\times Y$ we have
$H(x_n,y_n) \to H(x_0,y_0)$ be a finite application of the pasting
lemma.  If $\{C_{i_n}\}$ is infinite, then since $\{C_i\}$ is a null
sequence and $x_n\in C_{n_i}$, we have $C_{i_n}\to x_0$ and thus
$H(x_n,y_n)\in D_{i_n}\to z_0= H(x_0,y_o)$.  Thus $H$ is continuous
on all of $X\times  Y$.
\end{proof}

\begin{lem}\label{alem1}
If $A$ is an annulus and $C$ is a closed subset of $\int A$ that
separates the boundary components $J_1$ and $J_2$ of $A$, then
some component of $C$ separates $J_1$ from $J_2$ in $A$.
\end{lem}

This is a consequence of Phragm$\acute{e}$n-Brouwer properties. (see \cite{hw}.)

Let $B_r^X(x) = \{ y\in X \ | \ d(x,y)<r\}$ and $S_r^X(x)= \{ y\in X \ |\ d(x,y) = r\}$.  If $X$ is a planar set then $B_r^X(x) = B_r^\mathbb R(x)\cap X$ and $S_r^X(x) = S_r^\mathbb R(x)\cap X$.

\begin{lem}\label{alem2}
If $X$ is a planar set which is not semilocally simple connected at $x$, then there exists arbitrarily small $r$ such that $S_r^\mathbb R(x)\not\subset X$.
\end{lem}

\begin{proof}
Suppose that there exists an $\epsilon>0$ such that for all $r<\epsilon$, $S_r^\mathbb R(x)\subset X$.  Then $B^\mathbb R_\epsilon(x)\subset X$ which implies that $B^\mathbb R_\epsilon(x) = B^X_\epsilon(x)$ and thus $X$ is simply connected at $x$.
\end{proof}

\begin{lem}\label{homotoped off}
Let $X$ be a subset of $\mb R^2$ and $U$ a closed disk in the plane whose boundary is not contained in $X$. If $\alpha$ is an essential loop contained in $X\cap \int(U)$, then $\alpha$ cannot be freely homotoped out of $U\cap X$.
\end{lem}

\begin{proof}
Let $X$ be a subset of $\mb R^2$ and $U$ a closed disk in the plane whose boundary is not contained in $X$. Suppose $\alpha$ is an essential loop contained in $X\cap \int(U)$ such that $\alpha$ can be homotoped out of $U\cap X$.

Let $A$ be an annulus with boundary components $J_1$ and $J_2$. Then there exists a map $h: A\to X$ which takes $J_1$ to $\alpha$ and $J_2$ to a loop in $X$ not intersecting $U\cap X$. Then $h^{-1}(\partial U\cap X)$ is a closed subset contained in the $\int A$ which separates $J_1$ from $J_2$ in $A$. Hence by Lemma \ref{alem1}, some component $C$ of $h^{-1}(\partial U\cap X)$ separates $J_1$ from $J_2$ in $A$.  Since $\partial U \not\subset X$, $h\bigr|_{C}$ maps to an arc in $X$.  Hence by
Tietze Extension Theorem, we may alter $h$ to a map of a disk into $X$.  Therefore $\alpha$ is null-homotopic, a clear contradiction.

\end{proof}

\begin{lem}\label{cutoff}
Suppose the $f:\partial \mathbb D \to X$ is a continuous map into a planar set.  If $f$ is null-homotopic then there exists a map $h: \mathbb D \to X$ such that $\diam {h(\mathbb D)}\leq 2\diam{f(\partial\mathbb D)}$.
\end{lem}

\begin{proof}
Let $x_0\in f(\partial\mathbb D)$ and $r_0 = \diam{f(\partial\mathbb D)}$.  Let $r_1= \inf\{ r\geq r_0 \ | S_r^\mathbb R(x_0) \not \subset X \}$.  Then $S_{r_1}^X(x)$ is the disjoint union of closed intervals (some of which are possible degenerate).  Let $h: \mathbb D \to X$ be a null homotopy of $f$.   Consider the components of $h^{-1}\bigr(S_{r_1}^X(x)\bigl)$.   Then as before we can cut $h$ off at $S_{r_1}^X(x)$.   Then the modified $h$ can be contracted into $B_{r_0}^X(x_0)$ which completes the proof.

\end{proof}

\begin{lem}\label{alem3}
If $x_n \to x$ in a Peano continuum, then there exists a path
$T:I\to X$ with the property that $T(0)= x_1$, $T(1)= x$, and
$T(1-\frac1n)= x_n$.
\end{lem}

\begin{proof}
Let $U_n$ be a path connected neighborhood of $x$ contained in
$B_\frac1n(x)\cap U_{n-1}$.  Let $T_n:[1-\frac1n,
1-\frac1{n+1}]\to X$ be any path from $x_n$ to $x_{n+1}$ with
the property that if $x_n, x_{n+1}\in U_k$ then $T_n([1-\frac1n,
1-\frac1{n+1}])\subset U_k$. $T_n$ will exist since there is a
minimal $U_k$ containing both $x_n$ and $x_{n+1}$.  Then $T:I\to
X$ by $T(x) = T_n(x) $ for $x\in [1-\frac1n, 1-\frac1{n+1}]$ and
$T(1) =x$ is a continuous path.
\end{proof}

\end{document}